\documentclass[12pt]{amsart}

\usepackage{amsmath,amstext,amssymb,amsopn,amsthm}
\usepackage{url,verbatim}
\usepackage{mathtools}
\usepackage{enumerate}

\usepackage{color,graphicx}

\usepackage[margin=30mm]{geometry}
\usepackage{eucal,mathrsfs,dsfont}

\numberwithin{equation}{section}

\allowdisplaybreaks

\usepackage[utf8]{inputenc}
\usepackage{graphicx,amsmath,amssymb,enumerate,color,amsthm}
\usepackage{hyperref}
\usepackage{epstopdf}
\renewcommand{\P}{\mathbb{P}}

\newcommand{\F}{\mathcal{F}}

\newcommand{\M}{\mathbb{M}}
\newcommand{\R}{\mathbb{R}}
\newcommand{\N}{\mathbb{N}}

\newcommand{\E}{\mathbb{E}}
\newcommand{\1}{\mathbf{1}}
\newcommand{\ust}{\mathbf{u}}
\newcommand{\bs}{\mathbf{s}}
\newcommand{\bv}{\mathbf{v}}
\newtheorem{thm}{Theorem}[section]
\newtheorem{lemma}[thm]{Lemma}
\newtheorem{corol}[thm]{Corollary}
\newtheorem{propo}[thm]{Proposition}

\theoremstyle{definition}

\newtheorem{defin}[thm]{Definition}

\newtheorem{remark}[thm]{Remark}

\newcommand{\prt}{\partial}
\newcommand{\eps}{\varepsilon}
\newcommand{\unif}{\mathcal{U}}
\newcommand{\ren}{\mathbb{U}}
\newcommand{\wt}{\widetilde}
\newcommand{\cB}{\mathcal{B}}
\DeclareMathOperator{\leb}{Leb}

\title{Can one make a  laser out  of cardboard?}
\author{Krzysztof Burdzy and Tvrtko Tadi\'c}

\address{KB: Department of Mathematics, Box 354350, University of Washington, Seattle, WA 98195, USA}
\email{burdzy@math.washington.edu}

\address{TT: Microsoft Corporation (City Center Plaza Bellevue), One Microsoft Way, Redmond, WA 98052 \and Department of Mathematics, University of Zagreb, Bijenička cesta 30,
10000 Zagreb, Croatia}
\email{tvrtko@math.hr}

\thanks{KB: Research supported in part by NSF Grant DMS-1206276. TT: Research supported in part by Croatian Science Foundation grant 3526.}

\keywords{random reflections, stopped random walks, Wiener-Hopf equation,  undershoot, overshoot}
\subjclass[2010]{60G50, 60K05, 37D50, 37H99}
\begin{document}

\maketitle

\begin{abstract}
We consider two dimensional and three dimensional semi-infinite tubes
made of ``Lambertian'' material, so that the distribution of the direction of a reflected light ray has the density proportional to the cosine of the angle with the normal vector. If the light source is far away from the opening of the tube then the exiting rays are (approximately) collimated in two dimensions but are not collimated in three  dimensions.
An observer looking into the three dimensional tube will see ``infinitely bright'' spot at the center of vision. In other words, in three dimensions, the light brightness grows to infinity near the center as the light source moves away. 

\end{abstract}

\section{Introduction}

We will examine the behavior of light rays in semi-infinite tubes. The ``cardboard'' in the title of the paper refers to a material reflecting light according to the Lambertian distribution, to be described later in the introduction. The Lambertian distribution arises as the only physically possible reflection process in which reflected rays have random directions independent of the incidence angle (this follows from formula (2.3) in \cite{ABS}). The ``laser'' effect refers to a possible collimation of light rays exiting the tube. We will show that if light rays are released far from the end of the tube and they reflect according to the Lambertian distribution then the exiting rays are collimated in two dimensions but are not collimated in three  dimensions. So the answer to the question posed in the title is positive only in two dimensions. 

The three dimensional model does involve  a singularity but of a milder type. We will show that an observer looking into the tube will see ``infinitely bright'' spot at the center of vision. In other words, the light brightness grows to infinity near the center as the light source moves away. 

The present project is a prelude to the study of Lambertian reflections in fractal domains. Some fractal domains have narrow channels and one would like to know how light travels within such channels. This article analyzes a toy model for the light behavior in a long thin channel. In future articles, 
we plan to extend this direction of research to light reflections in thorns with smooth boundaries and, ultimately, thorns with fractal boundaries.

Our project is inspired by and related to a number of other projects. Lapidus and Niemeyer (\cite{LN3,LN1,LN2}) considered billiards with the specular (classical) reflection in fractal billiards.
Comets et al. (\cite{comets1,comets2,comets3})
studied  random Lambertian reflections in smooth domains with irregular shapes. 
Angel et al. (\cite{ABS}) showed that Lambertian reflectors could be approximated by deterministic reflectors.
Evans (\cite{evans}) studied a  model of stochastic billiards were the reflection angle was uniform.

We will  describe the asymptotic behavior (angle and position) of the light ray when it reaches the end of the tube
when the light source  is far away.
The  motion of light rays along the tube  is governed by a random walk. In order to 
find the exit position and angle of the light ray we need to find estimates for the distributions
of undershoot and overshoot of a symmetric random walk. We will derive a number of explicit formulas using the
Wiener-Hopf equation  
and various  results from \cite{asmussen,chow,doney,erickson,mikosch,rogozin,spitzer}.  
See the book by Kyprianou \cite{Kypr} for an introduction to the topic.

An intriguing and challenging aspect of the two dimensional model is that it leads to the ``critical'' case of the Central Limit Theorem. The  model is associated with a random walk with steps that do not have a finite variance but nevertheless the CLT holds (although we will not use this fact in our paper).

The rest of the paper is organized as follows. 
We will  present a more detailed overview of our main results in the next section. Section \ref{u22.3} contains a review of known results on random walks, Wiener-Hopf equation and related topics. We will derive there some new results needed later in the paper. Section \ref{j12.2} is devoted to the analysis of the two-dimensional model and finally Section \ref{sec:cylinder} presents results on the three dimensional tube.

\section{The model and main results}\label{j10.4}

We start with the description of Lambertian reflections of light.
A physical surface is Lambertian if its apparent brightness does not
depend on the angle at which the observer is looking at it.
The Moon, in its full moon phase, is approximately Lambertian because it appears to be
a globally flat surface to terrestrial observers despite being round.
Lambertian reflections are also known as the Knudsen Law in the theory of gases.
We will present the two-dimensional model in this section. See Section \ref{sec:cylinder} for the three-dimensional case.

Consider a set $D\subset \R^2$ with a smooth boundary.
Suppose a light ray hits a point $x\in \partial D$ and reflects. The outgoing light ray travels at an angle $\Theta$
with the inward normal vector at $x$. The direction of the outgoing light ray is independent of the direction of the incoming light ray.
 The  density  of $\Theta$ is given by (see Figure \ref{fig1}),
\begin{equation}
 f(\theta) = \begin{cases}
                         \frac{1}{2}\cos \theta, & \text{for  }\theta \in (-\pi/2,\pi/2), \\
			  0,&  \textrm{otherwise}.
                        \end{cases} \label{eq1}
\end{equation}

\begin{figure}[ht]
\begin{center}
 \includegraphics[width=5.5cm]{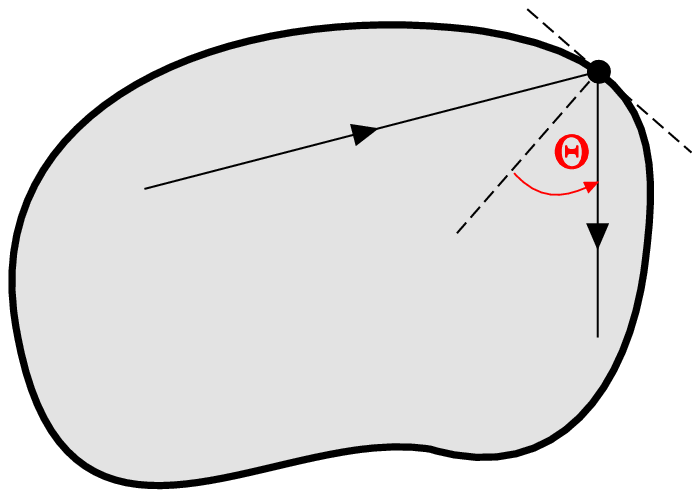}
\quad\includegraphics[width=5.5cm]{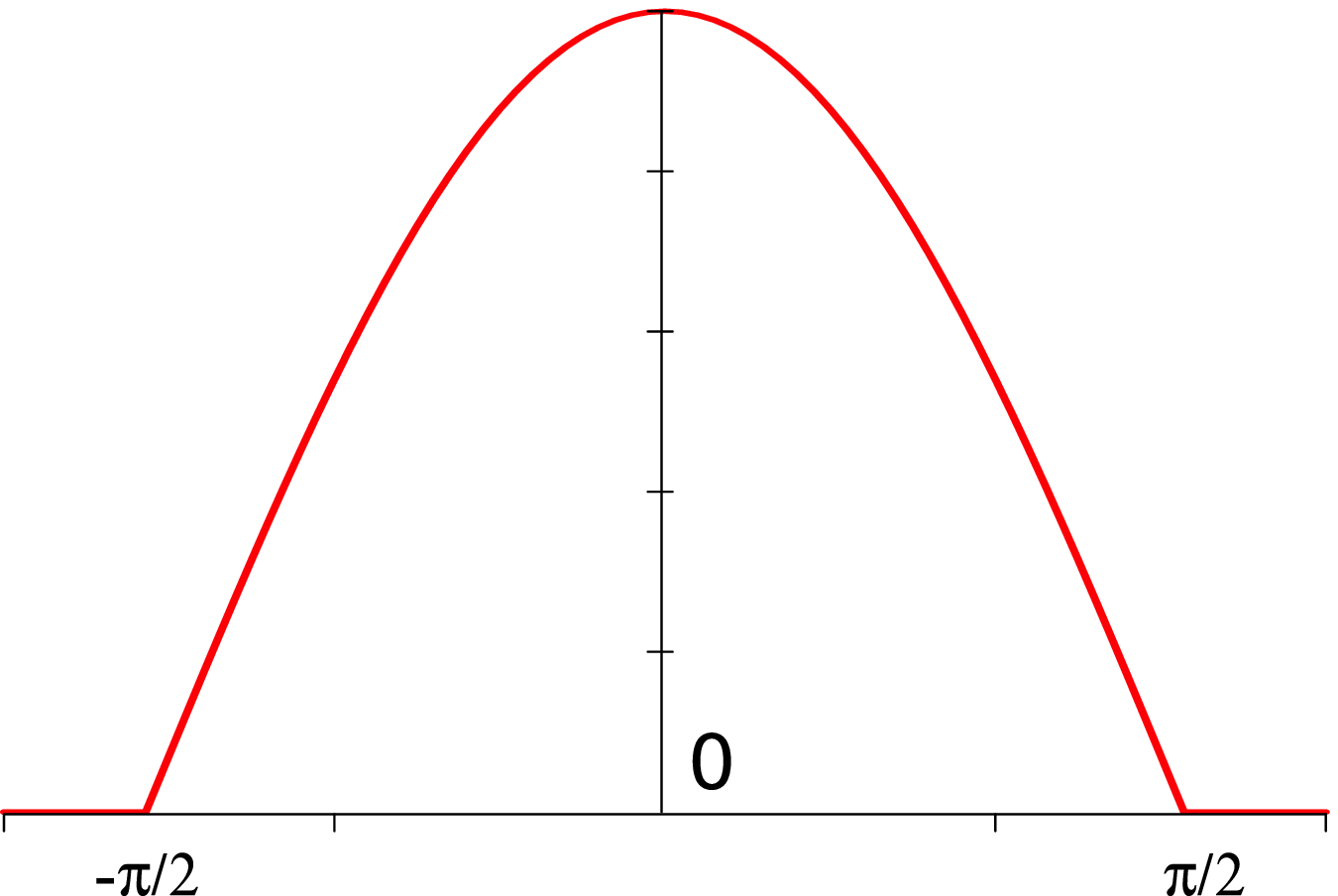} 
\\
\caption{Random reflection angle $\Theta$ and its density.}\label{fig1}
\end{center}
\end{figure}

The first part of the paper will be devoted to reflections in a semi-infinite
strip $D=\{(x,y)\in \R^2\ :\ x\leq 0,\; 0\leq y\leq 1\}$. We will assume that the light ray  starts at $(-s,0)$ for some $s>0$ and travels  in a  direction which forms  a random angle with the normal vector, with the  density  given by
$(\ref{eq1})$. The horizontal coordinate $s$ of the starting point will play the role of the main parameter in our model. 
Whenever the light ray hits the boundary of $ D$, it reflects according to the Lambertian scheme (see Figure \ref{fig:LPs}). In particular, all reflection angles are jointly independent. At a certain time, the light ray will exit the strip through its opening $\partial_rD:= \{(x,y)\in \R^2\ :\ x= 0,\; 0\leq y\leq 1\}$.
Let $(0,Y_s)$ be the exit point and let $\Lambda_s$ be the  exit angle (see Figure \ref{fig:LPs}). Our main result is concerned with the asymptotic behavior of the joint distribution of $(\Lambda_s, Y_s)$ as $s\to\infty$.

\begin{figure}[ht]
\begin{center}
 \includegraphics[width=11.5cm]{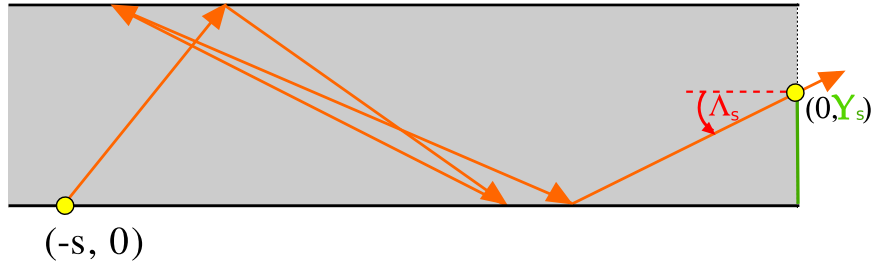}\\
\caption{Starting point $(-s,0)$, exit angle $\Lambda_s$ and exit location $(0,Y_s)$.}\label{fig:LPs}
\end{center}
\end{figure}

The $x$-coordinates of the points of reflection
constitute a random walk. A step of this random walk 
is a symmetric random variable $X_k$ satisfying $\P(X_k>x)\sim cx^{-2}$ for $x\to \infty$. 
An essential part of our analysis is devoted to ``undershoot'' and ``overshoot'', defined informally as follows.
The undershoot $U_s$ is the horizontal distance from the last reflection point to $\prt _r D$. The overshoot $O_s$ is the 
difference between the size of the random walk step that goes beyond 0 and $U_s$ (rigorous definitions will be given below).
One of our main results is the following simplified version of Theorem \ref{thm:asymindeplim},
$$\lim_{s\to \infty}\P\left(\sqrt{\frac{\log U_s}{\log s}}\leq t, \frac{U_s}{U_s+O_s}\leq u\right)=t u^2/2,
\qquad \text{  for  } t,u \in [0,1].
$$

Let $\unif[a,b]$ denote the uniform distribution on $[a,b]$.
Our basic result on
the limiting distribution for the exit angle $\Lambda_s$ and exit location $Y_s$, Theorem \ref{thm:unif_law}, says that, when $s\to\infty$,
$$(\Lambda_s,Y_s)\stackrel{d}{\to} (0,\unif[0,1]).$$
We use the results on  overshoot and  undershoot of the random walk to obtain  more accurate information on the joint distribution of $\Lambda_s$ and  $Y_s$
in Theorem \ref{thm:scaling}. For  $ t,u \in [0,1]$,
$$\lim_{s\to\infty}\P\left(\sqrt{\frac{\log \cot |\Lambda_s| }{\log s}}\leq u, I \Lambda_s <0, Y_s\leq t\right)
=\begin{cases}
	u(1-(1-t)^2)/2,& I=1,\\
	ut^2/2,& I=-1.
  \end{cases}
$$

At this point we can answer the question posed in the title of the paper. We place the eye of the observer at approximately $(0,y)$ (see Figure \ref{fig:eye_im_beta}). 
\begin{figure}[ht]
\begin{center}
 \includegraphics[width=8cm]{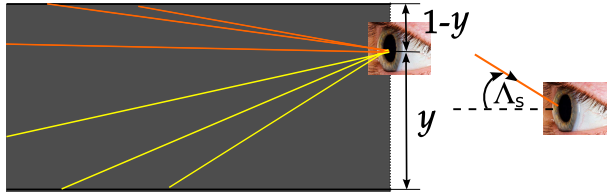}\\
\caption{Light rays arriving at the eye placed at $(0,y)=(0,2/3)$.}\label{fig:eye_im_beta}
\end{center}
\end{figure}
The distribution of the light rays arriving at the eye is expressed in terms of $\Lambda_s$ and given in Corollary
\ref{corol:eyedist} as follows.
For $u,y\in [0,1]$ and $\varepsilon\in (0,y)$ we have
$$\lim_{s\to\infty}\P\left(\sqrt{\frac{\log \cot |\Lambda_s| }{\log s}}\leq u, I\Lambda_s \leq 0\mid Y_s\in (y-\varepsilon,y ]\right)
=\begin{cases}
    u(1-y+\varepsilon/2),& I=1,\\
    u(y-\varepsilon/2),& I=-1.
\end{cases}
$$
The distribution is illustrated in Figure \ref{fig:angle_dist}. Note the asymmetric singularity at 0. We continue the discussion of the two-dimensional results in Section \ref{j12.1}.
\begin{figure}[ht]
\begin{center}
 \includegraphics[width=9cm]{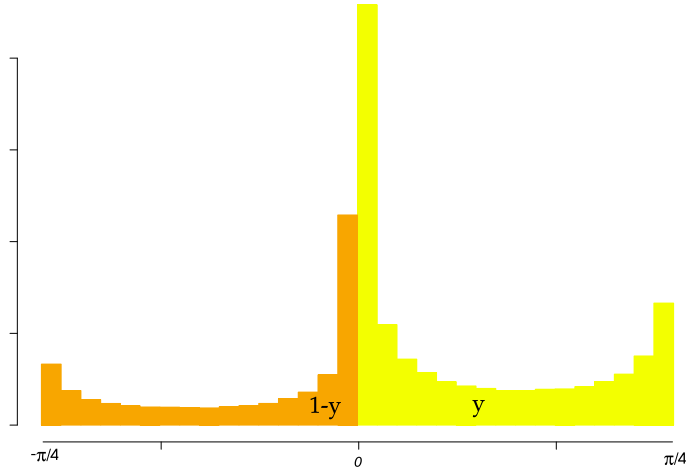}\\
\caption{Approximate distribution of $\Lambda_s$ given $Y_s=y$, with $s=1,000$ and $y=2/3$.}\label{fig:angle_dist}
\end{center}
\end{figure}

We will discuss the three dimensional case in Section \ref{sec:cylinder}.
The fundamental difference between two and three dimensional cases is that the asymptotic distribution of the direction of the light ray exiting the tube at a specific point is  degenerate in the two dimensional case and non-degenerate in the three dimensional case. We do not have an explicit formula for the asymptotic exit distribution in the three dimensional case but we have some estimates.
In three dimensions we have the following theorem of different nature.
Let $\bv_s  $ be the unit
vector representing the direction of the light ray at the exit time 
assuming it leaves the tube at the center of the opening (see Section \ref{sec:cylinder} for the rigorous definitions).
Let $\cB(r) = \{(x,y,z): x^2 + y^2 + z^2 = 1, y^2+ z^2 \leq r^2 , x>0 \}$ denote a ball on the unit sphere.
A somewhat informal statement of Theorem \ref{thm:u20_1} is
\begin{thm}
For any $0 < r_1 < r_2 < 1$,
\begin{align*}
\lim_{s\to \infty}
s \P\left(\bv_s \in \cB\left( \frac{r_2}{s}\right) \setminus \cB\left( \frac{r_1}{s}\right)\right) 
=    \frac {r_2-r_1}{2 \pi^2}.
\end{align*}
\end{thm}
Consider an observer at the center of the opening of the tube, looking towards the interior of the tube.
The theorem  says that small annuli at the center of the field of vision,
with the area of magnitude $1/s^2$, receive about $1/s$ units of light. Hence, the apparent brightness is about $s$ at the distance $1/s$ from the center, if the light source is $s$ units away from the opening. 
This means that the surface of the tube does not appear to be Lambertian, i.e., the surface does not have uniform apparent brightness. This can be explained by the fact that not all parts of the surface of the tube  receive the  same amount of light.

\section{Review of stopped random walks}\label{u22.3}

In this section we establish notation that will be used throughout the paper, give some rigorous definitions, recall some known results and derive some theorems on general random walks, not necessarily those arising in the random reflection model. 

We will study a random walk $\{S_n, n\geq 0\}$, with $S_0=0$ and $S_n=S_{n-1}+X_{n}$ for $n\geq 1$,
where $\{X_n, n\geq 1\}$ is an i.i.d. sequence.
We will always assume that $X_n$'s are continuous random variables. Some of the results stated in this paper might not be  true for lattice variables.

\subsection{Renewal measures and ladder processes}

 The ascending ladder epochs $\{T_k^+:k\geq 0\}$ are defined as
\begin{align}\label{j8.1}
T_0^+ &=0,\\
 T_n^+ &=\inf\left\{k>T_{n-1}^+:S_k>S_{T_{n-1}^+}\right\}, \qquad n\geq 1.\nonumber
\end{align}
It is easy to see  that 
$\{(T_{k}^+-T_{k-1}^+,S_{T_k^+}-S_{T_{k-1}^+})\ :\ k\geq 1\}$ is an i.i.d. sequence.
Let  $Z_k^+=S_{T_k^+}-S_{T_{k-1}^+}$ for $k\geq 1$. For $n=1,2,\ldots$ we call 
$H_n^+=S_{T_n^+}=\sum_{k=1}^nZ^+_k$ the ascending ladder heights.
\vspace{0.1cm}

Similarly, we define the descending ladder epochs $\{T_k^-:k\geq 0\}$ by setting
$T_0^-=0$ and $T_n^-=\inf\{k>T^-_{n-1}:S_k<S_{T^-_{n-1}}\}$ for $n\geq 1$.
The sequence
$\{(T_{k}^--T_{k-1}^-,S_{T_k^-}-S_{T_{k-1}^-})\ :\ k\geq 1\}$ is  i.i.d. 
We let $Z_k^-=S_{T_k^-}-S_{T_{k-1}^-}$ for $k\geq 1$ and call $H_n^-=S_{T_n^-}=\sum_{k=1}^nZ_{k}^-$, $n\geq 1$, the descending ladder heights.
\vspace{0.1cm}

The following result can be found in \cite{doney} (see relations (4a) and (4b)). A more general sufficient and necessary condition for the finiteness of  ladder step moments was given in \cite{chow}.
\begin{lemma}\label{lema:swh1}
Suppose that $\E(X_1)=0$. 
\begin{enumerate}[(a)]
 \item If $\E[(X_1^\pm)^2]<\infty$ then $\E[|Z_1^\pm|]<\infty$.
 \item $\E[X_1^2]<\infty$ if and only if $\E[Z_1^+]\E[-Z_1^-]<\infty$. Moreover,
$$\E[X_1^2]=2\E[Z_1^+]\E[-Z_1^-].$$
\end{enumerate}
\end{lemma}

This immediately implies the following corollary.

\begin{corol}
\label{corol:swh1}
If $X_1$ is a symmetric random variable
then $\E[X_1^2]<\infty$ if and only if 
$\E[Z_1^+]=\E[-Z_1^-]<\infty$. Moreover,
$2\E[Z_1^+]^2=\E[X_1^2]$.
\end{corol}

We define  renewal measures by
\begin{equation*}
 \ren^{\pm}(dx)=\sum_{k=1}^{\infty}\P(H_k^\pm\in\, dx).
\end{equation*}
One can show that for a measurable set $A\subset \R$ (see \cite[(2.4)]{asmussen}),
\begin{equation*}
 \ren^{\pm}(A)=\E\left(\sum_{k=0}^{T^\mp_1-1}\1_{\{S_k\in A\}}\right).
\end{equation*}
This formula can be written in the following way. For a Borel set $A\subset \R$,
\begin{align}
 \ren^-(A) &=\sum_{n=0}^{\infty}\P(S_{0}\leq 0, \ldots, S_{n-1}\leq 0, S_n\leq 0, S_n\in A).\label{j8.2}
\end{align}
\begin{lemma}\label{lema:3.3}
We have 
\begin{equation}
 \P(Z_1^+\geq t)=\int_{-\infty}^0\P(X_1>t-s) \ren^-(ds). \label{eq:z1+}
\end{equation}
\end{lemma}
\begin{proof} Using the definition of $H^+_1=Z_1^+$ and \eqref{j8.2}, we obtain 
\begin{align*}
 \P(Z_1^+>t) &= \P(S_1>t) +\P(S_1\leq 0, S_2 >t) +\P(S_1\leq 0,S_2\leq 0, S_3>t)+\ldots\\
 &= \int_{-\infty}^0\P(X_1>t-s) \times \\
& \quad \times\Big(\P(S_0\in ds)+\P(S_0\leq 0, S_1\in ds)+\P(S_0\leq 0, S_1\leq 0, S_2\in ds)+\ldots \Big)\\
 & =\int_{-\infty}^0\P(X_1>t-s) \ren^-(ds).
\end{align*}
\end{proof}

The following result is the well known renewal theorem (see \cite[Sect.~3.4]{durrett}; see \cite{erickson} for extensions).
\begin{thm}\label{thm:strrenthm}
For $\E[Z_1^+]\in(0,\infty]$ and all $h>0$,
$$\lim_{t\to\infty }\ren^+([0,t])/t= \frac{1}{\E[Z_1^+]} \quad \textrm{and}\quad \lim_{t\to\infty }\ren^+([t,t+h])=\frac{h}{\E[Z_1^+]}.
$$
\end{thm}
This implies that if $\E[Z_1^+]\in(0,\infty]$ then for all $h>0$,
\begin{align}\label{cor:bddinc}
\sup_{t\geq 0}\ren^+([t,t+h])<\infty.
\end{align}

\begin{defin}

(a)
For a function $r:[0,\infty)\to [0,\infty)$, let
\begin{align*}
I^h_+(r) &= h\sum_{k=1}^{\infty}\sup\{r(x) :\ x\in[(k-1)h,kh)\},\\
 I^h_-(r) &= h\sum_{k=1}^{\infty}\inf\{r(x) :\ x\in[(k-1)h,kh)\}.
\end{align*}
We say that  $r(x)$  is directly Riemann integrable (d.R.i.) if
$\lim_{h\to 0}I^h_+(r) =\lim_{h\to 0}I^h_-(r)$ and the limits are finite.

(b) Recall that the variation $V(r)$ of $r$ over $[0,\infty)$ is defined as
\begin{align*}
\sup \sum_{k=1}^n |r(x_k) - r(x_{k-1})|,
\end{align*}
where the supremum is taken over all sequences $0 \leq x_0 \leq x_1 \leq \dots \leq x_n$.
\end{defin}

\begin{remark}\label{j8.5}
(i)
It is elementary to check that
\begin{align*}
I^h_+(r) - I^h_-(r) \leq V(r) h.
\end{align*}
If $\int_0^\infty r(x) dx < \infty$, in the sense of the Lebesgue integral, then
it is easy to see that
\begin{align*}
I^h_+(r) &\leq \int_0^\infty r(x) dx + (I^h_+(r) - I^h_-(r))
\leq  \int_0^\infty r(x) dx + V(r) h,\\
I^h_-(r) &\geq \int_0^\infty r(x) dx - (I^h_+(r) - I^h_-(r))
\geq \int_0^\infty r(x) dx - V(r) h.
\end{align*}
This implies that if $V(r) < \infty$ and 
$\int_0^\infty r(x) dx < \infty$ then $r(x)$ is d.R.i.

(ii)
If $r:[0,\infty)\to[0,\infty)$ is decreasing 
then it has a bounded variation. Hence, if $r$ is decreasing and  $\int_0^\infty r(x) dx < \infty$ then $r(x)$ is d.R.i.

(iii) Every d.R.i. function is necessarily bounded. Otherwise we would have $I^h_+(r) =\infty$ for all $h>0$. 
\end{remark}

\begin{lemma}\label{j10.2}
Suppose that $r:[0,\infty)\to [0,\infty)$ is d.R.i. and $\E[Z_1^{+}]\in(0,\infty]$. 
Then 
\begin{equation}
 \sup_{s\geq 0}\int_0^{\infty} r(s+x)\ren^{+}(dx)<\infty,\label{thm:bnd1}
\end{equation}
\begin{equation}
 \lim_{s\to\infty}\int_0^sr(s-x)\ren^{+}(dx)=\frac{1}{\E[Z_1^+]}\int_0^{\infty} r(s)ds.\label{thm:rwn1}
\end{equation}
\end{lemma}

\begin{proof}
The claim  $(\ref{thm:rwn1})$ can be found in \cite[(4.9)]{durrett} 
or  \cite[Thm.~3]{erickson}).

For $(\ref{thm:bnd1})$ we fix $h\in(0,1)$ and let $M_1\in(0,\infty)$ be an upper bound for $r$ (see Remark \ref{j8.5} (iii)). 
By  \eqref{cor:bddinc}, there exists $M_2>0$ such that $\ren^+([t,t+h])<M_2$
for $t\geq 0$. Let $n_0=\lfloor s/h\rfloor+1$. Note that $0\leq n_0h-s< h$. We have
\begin{align*}
& \int_0^{\infty}r(s+x)\, \ren^+(dx) 
= \int_0^{n_0h-s}r(s+x)\, \ren^+(dx)+\sum_{k=0}^{\infty}\int_{(n_0+k)h-s}^{(n_0+k+1)h-s}r(s+x)\, \ren^+(dx)\\
&\leq  M_1\ren^+(0,n_0h-s)\\ 
&\quad + \sum_{k=0}^{\infty}\Big(\sup\{r(x):x\in [(n_0+k)h,(n_0+k+1)h)\}\\ &\qquad \qquad \times \ren^+([(n_0+k)h-s,(n_0+k+1)h-s])\Big)\\
&\leq M_1\ren^+([0,1])+(I^h_+(r)/h) M_2.
\end{align*}
The right hand side is finite and does not depend on $s$ so \eqref{thm:bnd1} is true.
\end{proof}

For $s>0$ we let
\begin{align}\label{eq:NsMk} 
N_s &= \inf\{n>0: S_n>s\},\\
O_s&=S_{N_s}-s, \qquad U_s=s-S_{N_{s}-1}.\nonumber
\end{align}
We call $O_s$ the overshoot and $U_s$ the undershoot 
of 
the random walk $S_n$ at $s$. We will also use the overshoot and undershoot of the ladder height process, defined by 
\begin{align*}
N_s^+ &= \inf\{n>0: H^+_n>s\},\\
O_s^+&=H^+_{N_s^+}-s, \qquad
U_s^+=s-H_{N_{s}^+-1}^+.
\end{align*}
It is easy to see that
\begin{equation}
O_s=O_s^+\quad \textrm{and}\quad  U_s^+\leq U_s. \label{ou:inq2} 
\end{equation}

\begin{lemma}\label{prop:UOplus}
If $\E[Z_1^+]=\infty$, then $O_s^+\to\infty$ and $U^+_s\to\infty$ in probability as $s\to \infty$.
\end{lemma}

\begin{proof}
We have
\begin{align*}
 \P(U_s^+\leq m) &= \sum_{k=1}^{\infty}\P(H^+_k>s,s-m<H^+_{k-1}\leq s)\\ & =\int_{s-m}^s\P(Z^+_1>s-y)\ren^+(dy)
\leq \ren^+([s-m,s]).
\end{align*}
The right hand side converges to 0 by Theorem \ref{thm:strrenthm}
so $U^+_s\to\infty$ in probability as $s\to \infty$.

A similar calculation yields
\begin{align}\label{j9.1}
\P(O_s^+<m) &= \sum_{k=1}^{\infty}\P(s<H^+_k\leq s+m,H^+_{k-1}\leq s)\\& = \int_{0}^s\P(s-y<Z^+_1\leq s-y+m)\ren^+(dy).\nonumber
\end{align}

Note that
\begin{align*}
\P(s<Z^+_1\leq s+m) \leq \P(km<Z^+_1\leq (k+1)m)  + \P((k+1)m<Z^+_1\leq (k+2)m)
\end{align*}
for $km < s \leq (k+1)m$. It follows that
\begin{align*}
&\int_0^\infty \P(s<Z^+_1\leq s+m)ds
= \sum_{k=0}^\infty \int_{km}^{(k+1)m} \P(s<Z^+_1\leq s+m)ds\\
& \leq
\sum_{k=0}^\infty \int_{km}^{(k+1)m} 
 \Big(\P(km<Z^+_1\leq (k+1)m)  + \P((k+1)m<Z^+_1\leq (k+2)m) \Big) ds\\
& =
\sum_{k=0}^\infty m
 \Big(\P(km<Z^+_1\leq (k+1)m)  + \P((k+1)m<Z^+_1\leq (k+2)m) \Big) 
\leq 2 m < \infty.
\end{align*}
In other words, the function $s\to \P(s<Z^+_1<s+m)$ is integrable over $[0,\infty)$.
Since
\begin{align*}
\P(s<Z^+_1<s+m) = \P(Z^+_1<s+m) - \P(s\geq Z^+_1),
\end{align*}
the function $s\mapsto \P(s<Z^+_1<s+m)$ is the difference of two monotone and bounded functions. It follows that this function has bounded variation. Since it is also integrable, it is d.R.i., by Remark \ref{j8.5} (i). Hence, by $(\ref{thm:rwn1})$,
\begin{align*}
\lim_{s\to \infty} \int_{0}^s\P(s-y<Z^+_1\leq s-y+m)\ren^+(dy)= 0.
\end{align*}
This and \eqref{j9.1} imply that $O^+_s\to\infty$ in probability as $s\to \infty$.
\end{proof}

\begin{lemma}\label{thm:OU} 
Make one of the following assumptions.
\begin{enumerate}[(i)]
 \item  $\E[X_1]=0$, $\E[(X_1^-)^2]<\infty$ and $\E[(X_1^+)^2]=\infty$.
 \item $X_1$ is symmetric and $\E(X_1^2)=2\E[(X_1^\pm)^2]=\infty$.
\end{enumerate}
 Then $O_s^+\to\infty$ and $U^+_s\to\infty$ in probability as $s\to \infty$.
\end{lemma}

\begin{proof}
Using Lemma \ref{lema:swh1}. for case (i), or Corollary \ref{corol:swh1} for (ii)  we obtain $\E[Z_1^+]=\infty$.
The claim follows from $(\ref{ou:inq2})$ and Lemma \ref{prop:UOplus}.
\end{proof}

For functions $f,g: (0,\infty) \to \R$, we will write $f\sim g$
if $ \lim_{x\to\infty} g(x)/f(x) =1$.

\begin{defin}
For a function $h:\R^+\to \R^+$ we say that it is regularly varying with exponent (index)
$\alpha$ if 
\begin{equation}
\lim_{t\to\infty}h(xt)/h(t)=x^{\alpha} \label{limit1} 
\end{equation}
 for $x>0$. A function $h$ is 
called slowly varying if $\alpha =0$.
\end{defin}

Recall that $h$ is a 
regularly varying function  with index $\alpha$ if and only if it is of the form $h(x) = x^\alpha L(x)$ where  $L$ is a slowly varying function.

The following two results can be found in  \cite[Thms.~6 and 7]{erickson}.

\begin{thm}\label{thm:O/s}
Suppose that $\P(H_1^+>t)=t^{-1}L(t)$, where $L(t)$ is a slowly varying function.
Then 
$O_s^+/s\to 0$ and $U^+_s/s\to 0$ in probability as $s\to \infty$.

\end{thm}

\begin{thm}\label{thm:lim_unif}
Suppose that $Z$ has the  distribution $\unif[0,1]$.
If $t\mapsto \P(H_1^+>t)$ is regularly varying with index $-1$ and $m(t)=\int_0^t\P(H_1^+>x)\,dx$ then
$$\left(\frac{m(O_t^+)}{m(t)},\frac{m(U_t^+)}{m(t)}\right)\stackrel{d}{\to}(Z,Z), \quad \textrm{as}\ t\to \infty.$$
\end{thm}

The following  theorem can be found in  \cite[Thm.~1.2.4]{mikosch}.

\begin{thm}\label{thm:limunf}
Suppose that $h:\R^+\to \R^+$ is regularly varying with index $\alpha<0$. Then for every $a>0$,
the limit in $(\ref{limit1})$ is uniform in $x\in[a,\infty)$.
\end{thm}
The following result, known as Potter's Theorem, can be found in \cite[Thm.~1.5.6]{regularVariation}.
\begin{thm}\label{thm:Potter}
 Suppose that $h:\R^+\to \R^+$ is regularly varying with index $\alpha$. Then for any chosen $\delta>0$ and $\varepsilon>0$
there exists $t_0>0$ such that 
$$\frac{h(y)}{h(x)}\leq \delta\max\left\{\left(\frac{y}{x}\right)^{\alpha+\varepsilon}, \left(\frac{y}{x}\right)^{\alpha-\varepsilon}\right\},$$
for all $t\geq t_0$.
\end{thm}

\begin{defin}
A random variable $X_1$ is called relatively stable if there exists a sequence of numbers 
$a_n>0$, $n=0,1,\ldots $, such that $S_n/a_n\to 1$ in probability as $n\to \infty$.
\end{defin}

The following Relative Stability Theorem from \cite[Thm.~2]{rogozin} (see also Sect. 8.8 in \cite{regularVariation}, especially Thm. 8.8.1) provides various characterizations of stable distributions.

\begin{thm}\label{thm:relstb1}
If $\P(X_1\geq 0)=1$ then the following claims are equivalent. 
\begin{enumerate}[(a)]
 \item $O_x/x\to 0$ in probability as $x\to\infty$;
\item $U_x/x\to 0$ in probability as $x\to\infty$;
\item $\int_0^x\P(X_1\geq y)dy\sim L(x)$ where $L(x)$ is a slowly varying function;
\item $\ren^+([0,x))=\sum_{n=1}^{\infty}\P(S_n<x)\sim x/L(x)$ where $L(x)$ is the same function as in (c);
\item $X_1$ is relatively stable.
\end{enumerate}
\end{thm}

The following result is taken from \cite[Thm.~9]{rogozin}.

\begin{thm}\label{thm:relstb2}
Suppose that $S_n/a_n$ converge in distribution to a stable law with index $\alpha\leq 2$ for some sequence $a_n$.
If $\alpha \P(X_1>0)=1$ then $Z^+_1$ is relatively stable.
\end{thm}

Theorems \ref{thm:relstb1} and \ref{thm:relstb2} give the following result. 

\begin{corol}\label{corol:relstb}
If $X_1$ is symmetric and $S_n/a_n$ converges to a normal distribution for some sequence $a_n$ then 
\begin{enumerate}[(a)]
 \item $Z_1^+$ and $Z_1^-$ are relatively stable;
 \item $\ren^+([0,x)) = \ren^-((-x,0])\sim x/L(x)$ where $L(x)$ is a slowly varying function;
 \item $\int_0^x\P(Z_1^+\geq y)dy=\int_0^x\P(-Z_1^-\geq y)dy\sim L(x)$
where $L(x)$ is the same function as in (b).
\end{enumerate}
\end{corol}

A sufficient condition for the convergence of $S_n/a_n$  to a normal distribution
is contained in the following  very general theorem (see \cite[Cor.~1.4.8]{mikosch}).
\begin{thm}\label{j9.2}
Let $\P(|X_1|>t)=t^{-2}L(t)$ where $L$ is a slowly varying function.
Then $X_1$ is in the domain of the attraction of the normal distribution.
\end{thm}

\begin{lemma}\label{lem:list:1}
Suppose that $X_1$ is a continuous symmetric random variable and
 $t\mapsto \P(X_1>t)$ is regularly varying with index $-2$. The following
claims hold:
\begin{enumerate}[(a)]
 \item $\E[X_1^2]=\infty$ and $\E[Z_1^+]=\infty$;
 \item $\ren^+([0,t))\sim t/L(t)$ where $L(t)=\int_0^t\P(Z_1^+>x)\, dx$ is slowly
varying.
 \item $\ren^+(t\, ds)/\ren^+([0,t))$ converges weakly to the Lebesgue measure on $[0,A]$
for any $A>0$.
 \item The following limit  is uniform in $s\geq 0$,
\begin{align*}
\lim_{t\to\infty}\frac{\P(X_1>t(1+s))}{\P(X_1>t)}
=\frac{1}{(1+s)^2}.
\end{align*}
Moreover, for  every $\varepsilon>0$ there exists $t_0>0$ such that
\begin{align}\label{j9.3}
\frac{\P(X_1>t(1+s))}{\P(X_1>t)}\leq \frac{1}{(1+s)^{2-\varepsilon}},
\end{align}
for all $t\geq t_0$.
\end{enumerate}
\end{lemma}

\begin{proof}
(a) 
It is easy to check that if $t\mapsto \P(X_1>t)$ is regularly varying with index $-2$ then $\E[X_1^2]=\infty$. Part (a) then follows from
 Corollary \ref{corol:swh1}. 

(b) By Theorem \ref{j9.2}, the assumptions of  Corollary \ref{corol:relstb}
are satisfied. We can take the same slowly varying function function $L(t)=\int_0^t\P(Z_1^+>x)\, dx$  in
Theorem \ref{thm:relstb1} (c) and Corollary \ref{corol:relstb} (except in this case the first positive step is denoted  $Z_1^+$).
Part (b) of the lemma now follows from Corollary \ref{corol:relstb} (b).

(c) It follows from Corollary \ref{corol:relstb} (b) that
$$\frac{\ren^+(t[0,a])}{\ren^+([0,t))}
\sim \frac{at/L(at)}{t/L(t)}\sim  a,$$
for all $a\in [0,A]$.  This implies part (c) of the lemma.

(d) 
The claims follow from Theorems \ref{thm:limunf} and  \ref{thm:Potter}.
\end{proof}

\begin{lemma}\label{lema:lstep}
If $X_1$ is a continuous symmetric random variable and
 $t\mapsto \P(X_1>t)$ is regularly varying with index $-2$ then
\begin{equation}
 \lim_{t\to\infty}\frac{\P(Z_1^+>t)}{\P(X_1>t)\ren^+([0,t))}=1.
\end{equation}
\end{lemma}

\begin{proof}

From formula (\ref{eq:z1+}) we have 
\begin{align*}
\P(Z_1^+\geq t)&=\int_0^{\infty}\P(X_1>t+s) \ren^+(ds)\\
&=\P(X_1>t)\ren^+([0,t))\int_0^{\infty}\frac{\P(X_1>t+s)}{\P(X_1>t)} \frac{\ren^+(ds)}{\ren^+([0,t))}.
\end{align*}
The change of variable $s=tu$  gives us
\begin{align*}
\frac{\P(Z_1^+\geq t)}{\P(X_1>t)\ren^+([0,t))} &=\int_0^{\infty}\frac{\P(X_1>t(1+u))}{\P(X_1>t)} \frac{\ren^+(t\, du)}{\ren^+([0,t))}.
\end{align*}
It will suffice to show that 
$$\lim_{t\to\infty}\int_0^{\infty}\frac{\P(X_1>t(1+u))}{\P(X_1>t)} \frac{\ren^+(t\, du)}{\ren^+([0,t))}=\int_0^{\infty}\frac{1}{(1+s)^2}\,ds=1.$$
It is not hard to see, using Lemma \ref{lem:list:1} (c) and (d) that 
\begin{equation}
 \lim_{t\to\infty}\int_0^{A}\frac{\P(X_1>t(1+u))}{\P(X_1>t)} \frac{\ren^+(t\, du)}{\ren^+([0,t))}=\int_0^{A}\frac{1}{(1+s)^2}\,ds. \label{eq:concv1}
\end{equation}

According to \eqref{j9.3}, we can
pick $t_0>0$ such that $\P(X_1>t(1+s))/\P(X_1>t)\leq H(s):=(1+s)^{-3/2}$ for $t\geq t_0$.
For $A>t_0$, using the integration by parts formula,
\begin{align}
&\int_{A}^{\infty}\frac{\P(X_1>t(1+u))}{\P(X_1>t)} \frac{\ren^+(t\, du)}{\ren^+([0,t))}\leq \int_{A}^{\infty}H(u) \frac{\ren^+(t\, du)}{\ren^+([0,t))}\nonumber \\
&= \left. H(u)\frac{\ren^+([0,tu))}{\ren^+([0,t))}\right|_{u=A}^{u=\infty} -\int_{A}^{\infty} \frac{\ren^+([0,tu))}{\ren^+([0,t))}H'(u)\, du. \label{eq:int_bnd}
\end{align}
Recall the slowly varying function function $L(t)=\int_0^t\P(X_1>x)\, dx$
and use Lemma \ref{lem:list:1} (b)
to find $t_0$ such that 
$$\frac{t}{2L(t)}\leq \ren^+([0,t))\leq \frac{2t}{L(t)},$$
for $t\geq t_0$. From the fact that $L(t)$ is increasing for $t\geq A$ we obtain
$$\frac{\ren^+([0,ts))}{\ren^+([0,t))}\leq 4\frac{L(t)}{L(ts)}s\leq 4s.$$
Since $H(s)\sim \frac{1}{s^{3/2}}$ and $H'(s)\sim \frac{-1}{s^{5/2}}$, $(\ref{eq:int_bnd})$ is bounded by
$$ 4AH(A)-\int_{A}^{\infty} 4sH'(s) du.$$
Choose an arbitrarily small $\eps>0$ and pick $A$ large enough so that 
\begin{align*}
\int_{A}^{\infty}\frac{\P(X_1>t(1+u))}{\P(X_1>t)} \frac{\ren^+(t\, du)}{\ren^+([0,t))}<\varepsilon
\qquad \text{  and  }\qquad
\int_{A}^\infty\frac{1}{(1+s)^2}\, ds<\varepsilon.
\end{align*} 
It follows from this and \eqref{eq:concv1} that for large $t$,
\begin{align*}
&\left|\int_0^{\infty}\frac{\P(X_1>t(1+u))}{\P(X_1>t)} \frac{\ren^+(t\, du)}{\ren^+([0,t))}-1\right|\\
&\leq
\left|\int_0^{A}\frac{\P(X_1>t(1+u))}{\P(X_1>t)} \frac{\ren^+(t\, du)}{\ren^+([0,t))}-\int_0^{A}\frac{1}{(1+s)^2}\, ds\right|\\
&\quad +\int_{A}^\infty\frac{1}{(1+s)^2}\, ds
+
\int_{A}^\infty\frac{\P(X_1>t(1+u))}{\P(X_1>t)} \frac{\ren^+(t\, du)}{\ren^+([0,t))}
\leq \eps+\eps +\eps.
\end{align*}
The claim now follows.
\end{proof}

\begin{corol}\label{corol:regldst}
Suppose that $X_1$ is a continuous symmetric random variable and
 $t\mapsto \P(X_1>t)$ is regularly varying with index $-2$. Then $t\mapsto \P(Z_1^+>t)$
is regularly varying with index $-1$.
\end{corol}

\begin{proof}
The corollary follows from Lemmas \ref{lem:list:1} (b) and \ref{lema:lstep}.
\end{proof}

\begin{lemma}
Assume that $X_1$ is a continuous symmetric random variable and
 $t\mapsto \P(X_1>t)$ is regularly varying with index $-2$. Then 
\begin{align}\label{j9.4}
m(t):=\int_0^t\P(Z_1^+>x)\, dx\sim \sqrt{\int_0^{t}2xP(X_1>x)\, dx}.
\end{align}
\end{lemma}

\begin{proof}
Using Lemmas \ref{lem:list:1} (b) and \ref{lema:lstep} we get
\begin{equation}
 \lim_{t\to \infty}\frac{\P(Z_1^+>t)\int_0^t\P(Z_1^+>x)\, dx}{tP(X_1>t)}=1. \label{eq:lim2}
\end{equation}

 Since $\E[X_1^2]=\infty$, Corollary \ref{corol:swh1} yields
$E[Z_1^+]=\infty$. Therefore, $m(t)\uparrow \infty$ and $n(t):=\int_0^{t}xP(X_1>x)\, dx\uparrow \infty $ 
as $t\to\infty$. Hence, we can use l'Hopital's rule to calculate the limit 
$\lim_{t\to \infty}m^2(t)/n(t)$, and we get 
$$\lim_{t\to\infty}\frac{\left(\int_0^t\P(Z_1^+>x)\, dx\right)^2}{\int_0^{t}xP(X_1>x)\, dx}=\lim_{t\to\infty}\frac{2m'(t)m(t)}{n'(t)} = 2.$$
The last equality follows from \eqref{eq:lim2}. This easily implies the lemma.
\end{proof}

The lemma easily implies the following corollary.

\begin{corol}\label{corol:mlog} 
Suppose that $X_1$ is a continuous symmetric random variable such that $\lim_{t\to\infty}t^2\P(X_1>t)=c\in (0,\infty)$.
Then $$\P(Z_1^+>t)\sim \frac{\sqrt{c/2}}{t\sqrt{\log t}}\quad \textrm{and}\quad \int_0^t\P(Z_1^+>x)\, dx\sim \sqrt{2c}\sqrt{\log t}.$$
\end{corol}

\begin{lemma}\label{thm:OU/t} Suppose that $X_1$ is a symmetric random variable such that  
$t\mapsto P(X_1>t)$ is regularly varying with index $-2$. 
\begin{enumerate}[(a)]
 \item $O_s/s\to 0$ in probability when $s\to\infty$.
 \item If $m(t)=\sqrt{\int_0^t2x\P(X_1>x)\, dx}$
then $m(O_t)/m(t)\to \unif[0,1]$ in distribution when $t\to\infty$.
\end{enumerate}
\end{lemma}

\begin{proof}
(a) Note that $H^+_1 = Z^+_1$.
Part (a)
follows from Corollary \ref{corol:regldst}, \eqref{ou:inq2}  and Theorem \ref{thm:O/s}.

(b) 
Once again, we will use the fact that $H^+_1 = Z^+_1$.
Part (b)
follows from \eqref{ou:inq2}, \eqref{j9.4}, Corollary \ref{corol:regldst} and Theorem \ref{thm:lim_unif}.
\end{proof}

\begin{lemma}\label{prop:limit_log}
If $X_1$ is a symmetric random variable such that $\lim_{t\to\infty}t^2\P(X_1>t)=c\in (0,\infty)$
then $\sqrt{\log O_t/\log t}\to \unif[0,1]$ in distribution as $t\to\infty$.
\end{lemma}
\begin{proof}
The lemma follows from Corollary \ref{corol:mlog} and Lemma \ref{thm:OU/t} (b).
\end{proof}

\subsection{Wiener-Hopf equation}\label{sec:WH}
The {Wiener-Hopf integral equation} is 
\begin{equation}
 W(s)=g(s)+\int_{-\infty}^sW(s-y)F(dy), \label{swh:eq1}
\end{equation}
where $W:[0,\infty)\to \R$ is an unknown function. The function
$g:[0,\infty)\to \R$ and the probability distribution $F$ on $\R$
are given. 

We will make the following assumptions, common in this context.
\begin{itemize}
 \item $g(x)\geq 0$ for all $x\geq 0$ and $\sup_{x\in [0,a]}g(x)<\infty$ for all $a\geq 0$.
 \item $F$ is a probability measure with a well defined mean.
 \item We will consider only positive solutions to $(\ref{swh:eq1})$, i.e., $W(x)\geq 0$
for all $x\geq 0$.
\end{itemize}

If $g\equiv 0$, then we call the equation \emph{homogeneous}. Spitzer has shown in \cite{spitzer}
that, in general, there is no uniqueness for solutions to the homogeneous equation. However, uniqueness holds
if $F$ is concentrated on $[0,\infty)$; see \cite{durrett}.

In this paper, $F$ in \eqref{swh:eq1} will be the distribution of $X_1$. 
For $s\geq 0$ we define
\begin{align}
 M_k&=\max\{S_j: 0\leq j\leq n\}, \nonumber \\
 W_{\min}(s)&=\sum_{k=0}^{\infty}\E(g(s-S_k)\1_{(M_k\leq s)}).\label{swh:eq2}
\end{align}
We will need the following result from \cite[Cor.~3.1]{asmussen}.
\begin{thm}\label{j10.1}
Any solution $W$ of the equation $(\ref{swh:eq1})$ is of the form 
$W=W_{\min}+W_0$, where $W_0$ is a solution to the homogeneous equation. The function $W_{\min}(s)$
defined in $(\ref{swh:eq2})$ is the minimal solution. 
\end{thm}

Once again we quote a result from  \cite[Prop.~3.3]{asmussen}.
\thm{\label{whe:ms}For $W_{\min}$ we have 
$$W_{\min}(s)=\int_{0}^sG(s-x)\, \ren^+(dx),$$
where $G(s)=\int_{-\infty}^0g(s-y)\,\ren^-(dy)$.
}\rm\vspace{0.2cm}

\begin{lemma}\label{thm:whems}
Let $F$ be the probability distribution function of a symmetric random variable 
such that
$$\int_{-\infty}^{\infty}|x|F(dx)<\infty, \quad \int_{-\infty}^{\infty}|x|^2F(dx)=\infty,$$
and assume that for all $s\geq 0$,
$$g(s)\leq d(s)r(s),$$
where $d$ and $r$ are directly Riemann integrable and $d$ is non-increasing.
Then, $W_{\min}$ the minimal solution to the equation \eqref{swh:eq1}
has the property that 
$$\lim_{s\to\infty}W_{\min}(s)=0.$$
\end{lemma}

\begin{proof}

Let $\{X_k\}$ be a sequence of i.i.d. random variables with distribution $F$.
Since $X_1$ is a symmetric random variable,  $\ren^+(dx)=\ren^-(-dx)$ and we set 
$\ren=\ren^+$.  
Using the notation of Theorem \ref{whe:ms},
$$G(s)=\int_{-\infty}^0g(s-y)\ren^-(dy)\leq d(s)\int_0^{\infty}r(s+y)\ren(dy)\leq Cd(s),$$
where $C:=\sup_{s\geq0}\int_0^{\infty}r(s+y)\ren(dy)$. 
It follows from our assumptions and Corollary \ref{corol:swh1} that $\E Z^+_1=\infty$. Hence, we can apply \eqref{thm:bnd1} to see that
the constant $C$ is finite.
By \eqref{thm:rwn1} and Theorem \ref{whe:ms},
$$
 0\leq W_{\min}(s)=\int_0^{s}G(s-x)\,\ren(dx)\leq C\int_0^{s}d(s-x)\,\ren(dx) \to 0,
\qquad \text{  when  } s \to \infty.
$$
\end{proof}

\begin{corol}\label{cor:WHeq}
Suppose that $F$ is the probability distribution function of a symmetric random variable and $1-F(x)$ is regularly varying with index $-2$. 
Assume that there exist $C>0$ and $\alpha>0$ such that
\begin{equation}
 g(s)\leq \frac{C}{1+s^{2+\alpha}}.\label{inq:bound:g}
\end{equation}
Then $\lim_{s\to\infty}W_{\min}(s)=0$, where $W_{\min}$ is the minimal solution to the equation $(\ref{swh:eq1})$.
\end{corol}

\begin{proof}

Let $d(s) = r(s) = \frac{\sqrt{2C}}{1+s^{1+\alpha/2}}$. Then for all $s>0$,
\begin{align*}
 g(s)\leq \frac{C}{1+s^{2+\alpha}}
\leq \left(\frac{\sqrt{2C}}{1+s^{1+\alpha/2}}\right)^2
= d(s)r(s).
\end{align*}
Since 
$s\mapsto \frac{\sqrt{2C}}{1+s^{1+\alpha/2}}$ is a decreasing and Lebesgue integrable function on $[0,\infty)$,  it is directly
Riemann integrable, by Remark \ref{j8.5} (ii). The claim now follows from Lemma \ref{thm:whems}.
\end{proof}

\begin{remark}

Corollary \ref{cor:WHeq} may not hold for $\alpha=0$, as the following example shows.
Let $F$ be the cumulative distribution function of a symmetric random variable
with $1-F(x)=\frac{1/2}{1+x^2}$ for $x>0$. Let $N_s$ denote the stopping time defined in $(\ref{eq:NsMk})$ for the random walk with the step distribution $F$.
In this case we have $S_n/n\to 0$ a.s., and by the Chung-Fuchs Theorem the random walk is recurrent, hence $N_s<\infty$, a.s.
By Theorem \ref{j10.1}, the equation $(\ref{swh:eq1})$
with $g(s)=1-F(s)$ for $s\geq 0$ has the minimal solution
\begin{align*}
 W_{min}(s)&=1=\P(N_s<\infty)=\sum_{n=1}^{\infty}\P(N_s=n)
=\sum_{k=1}^{\infty}\P(X_k>s-S_{k-1}, M_{k-1}<s)\\
&=\sum_{k=1}^{\infty}\E(g(s-S_{k-1})\1_{(M_{k-1}<s)}).
\end{align*}

\end{remark}

\section{Two-dimensional model}\label{j12.2}

Recall the two dimensional model from Section \ref{j10.4}.
First we will review some  properties  of the random angle $\Theta$ with the density function given by $(\ref{eq1})$.
The cumulative distribution function $F_\Theta$ is equal to
\begin{align*}
F_{\Theta}(t) = \P(\Theta \leq t) = 
\begin{cases}
			  0, & t\leq \pi/2,\\
                         \frac{1}{2}(\sin t+1), & t \in (-\pi/2,\pi/2), \\
			  1,&t \geq \pi/2.
                       \end{cases}
\end{align*}
Note that $F_{\Theta}^{-1}:(0,1)\to (-\pi/2,\pi/2)$ is given by 
\begin{equation*}
 F_{\Theta}^{-1}(y)= \sin^{-1}(2y-1).
\end{equation*}
If $V$ has the distribution $ \unif(-1,1)$ then it is easy to check that the following  equalities hold
in the sense of distribution,
\begin{align}\label{lem:theta}
\sin\Theta = V,\
 \cos\Theta = \sqrt{1-V^2},\
 \tan \Theta= \frac{V}{\sqrt{1-V^2}}.
\end{align}

\begin{figure}[ht]
\begin{center}
 \includegraphics[width=6cm]{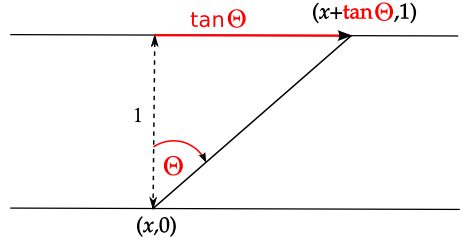}\\
\caption{Step of the random reflection.}\label{fig:rw_step}
\end{center}
\end{figure}

In the random reflection model described in Section \ref{j10.4}, if the ray is reflected at the point $(x,u)$ then its next reflection point will be at
$(x+\tan \Theta,1-u)$, where $u\in \{0,1\}$ and $\Theta$
has the  density given by $(\ref{eq1})$ (see Figure \ref{fig:rw_step}).

Let $\{\Theta_n\}$ be a sequence of i.i.d. random variables 
with density given by $(\ref{eq1})$ and set $X_n=\tan \Theta_n$. We define a random walk by setting $S_0=0$ and $S_n=S_{n-1}+X_n$ for $n\geq 1$. 
Recall that $N_s = \inf\{n>0: S_n>s\}$.
Then the trajectory of the light ray described in Section \ref{j10.4} consists of
\begin{enumerate}[(i)]
 \item line segments $[(-s+S_{n-1},(1-(-1)^{n-1})/2),(-s+S_{n},(1-(-1)^{n})/2)]$ for $n< N_s$;
 \item line segment between $(-s+S_{N_s-1},(1-(-1)^{N_s-1})/2)$ and  $(-s+S_{N_s},(1-(-1)^{N_s})/2)$.
\end{enumerate}
In view of \eqref{lem:theta},
the representation (i)-(ii) given above can start alternatively with 
 a sequence $\{V_n\}$ of i.i.d. $\unif(-1,1)$ random variables and  $X_n=V_n / \sqrt{1-V_n^2}$. 

\begin{defin}\label{def:LPs}
We define  $\Lambda_s$ to be the angle between the exiting ray given in (ii) and  the inward normal
to the right edge $[(0,0),(0,1)]$. We let $Y_s$  denote the $y$-coordinate of the point where the ray exits the tube through the right edge
(see Figure \ref{fig:LPs}).
\end{defin}

\begin{lemma}\label{lema:l1}
\begin{enumerate}[(a)]
 \item  The cumulative distribution 
function of $X_1$ is $F_{X_1}(x)=\frac{1}{2}+\frac{x}{2\sqrt{1+x^2}}$ and its density is
$f_{X_1}(x) = (1+x^2)^{-3/2}/2$ for $x\in\R$.
 \item  $\E[X_n]=0$ and $\E[X_n^2]=\infty$.
 \item The random walk $\{S_n\}$ is neighborhood recurrent.
\end{enumerate}
\end{lemma}
\begin{proof}
Part (a) follows from an elementary calculation.

(b) Since
$$\E|X_n|=
\E \left(\frac{|V_n| }{ \sqrt{1-V_n^2}}\right)
=\int_{-1}^{1}\frac{|x|}{\sqrt{1-x^2}}\frac{1}{2}\, dx= \int_{0}^{1}\frac{x}{\sqrt{1-x^2}}\, dx=1,$$
we must have $\E X_n=0$ by symmetry. The  second moment is infinite because
$$\E[X_n^2]=\int_{-\infty}^{\infty}\frac{2x^2}{(1+x^2)^{3/2}} dx=\infty.
$$ 

(c) Since $\E X_1=0$,   the strong law of large numbers shows that 
$S_n/n\to 0$, a.s.
This also holds in probability so the Chung-Fuchs Theorem  for random walks
implies that $\{S_n\}$ is a neighborhood recurrent random walk
(see \cite[Thms.~4.2.1 and 4.2.7]{durrett}). 
\end{proof}

It follows from Lemma \ref{lema:l1}(c)
that the light ray will  hit the line $\{x=0\}$, a.s.
In other words, with probability 1, the light ray will
exit the tube (strip) through the right edge.

\begin{lemma}\label{lem:ineq}
For $s>0$ and $t_1>t_2>0$ we have
\begin{equation}
\frac{\P(X_1>s/t_1)}{\P(X_1>s/t_2)}< \frac{t_1^2}{t_2^2}\quad \textrm{and}\quad \lim_{s\to \infty}\frac{\P(X_1>s/t_1)}{\P(X_1>s/t_2)}=\frac{t_1^2}{t_2^2}.
\label{eq2}
\end{equation}
Moreover, for all $t\in (0,1]$ and all $s>0$ we have
\begin{equation}
 0\leq \P(X_1>s/t)-t^2\P(X_1>s) \leq \frac{4}{(1+s^{2})^2}.\label{eq3}
\end{equation}
\end{lemma}

\begin{proof}
Both claims in $(\ref{eq2})$ follow easily  from the following identity,
\begin{align*}
\frac{\P(X_1>s/t_1)}{\P(X_1>s/t_2)}&=\frac{\frac{1}{2}-\frac{s}{2\sqrt{t_1^2+s^2}}}{\frac{1}{2}-\frac{s}{2\sqrt{t_2^2+s^2}}}=\frac{\frac{t_1^2}{\sqrt{t_1^2+s^2}(\sqrt{t_1^2+s^2}+s)}}{\frac{t_2^2}{\sqrt{t_2^2+s^2}(\sqrt{t_2^2+s^2}+s)}}\\
&=\frac{t_1^2}{t_2^2}\cdot\frac{\sqrt{t_2^2+s^2}(\sqrt{t_2^2+s^2}+s)}{\sqrt{t_1^2+s^2}(\sqrt{t_1^2+s^2}+s)}.
\end{align*}

We have
\begin{align}
& \P(X_1>s/t)-t^2\P(X_1>s)\nonumber\\
 &= t^2\left(\frac{1}{\sqrt{t^2+s^2}(\sqrt{t^2+s^2}+s)}-\frac{1}{\sqrt{1+s^2}(\sqrt{1+s^2}+s)} \right)\nonumber\\
&=t^2(1-t^2)\frac{1+\frac{s}{\sqrt{t^2+s^2}+\sqrt{1+s^2}}}{(t^2+s^2+s\sqrt{t^2+s^2})(1+s^2+s\sqrt{1+s^2})}.\label{eqn:imtbound}
\end{align}
Clearly, the expression in $(\ref{eqn:imtbound})$ is non-negative for $t\in (0,1]$ and $s>0$, $t^2(1-t^2)\leq 1$ and since
$s>0$ we have $1+\frac{s}{\sqrt{t^2+s^2}+\sqrt{1+s^2}}\leq 2$. For $s>1$,
$$
 \frac{1}{(t^2+s^2+s\sqrt{t^2+s^2})(1+s^2+s\sqrt{1+s^2})}\leq \frac{1}{(1+s^2)^2},
$$
so by $(\ref{eqn:imtbound})$ we get $\P(X_1>s/t)-t^2\P(X_1>s)\leq \frac{2}{(1+s^2)^2}$.
For $s\leq 1$ we have 
\begin{align*}
\P(X_1>s/t)-t^2\P(X_1>s)\leq \P(X_1>s/t) \leq 1\leq \frac{4}{(1+s^2)^2}.
\end{align*}
In either case, $(\ref{eq3})$ holds.
\end{proof}

Since $\E[X_1^2]=\infty$, we conclude from \eqref{ou:inq2} and Lemma \ref{thm:OU} 
that, in probability, when $s\to \infty$,
\begin{equation}
 O_s=S_{N_s}-s\to\infty\ \textrm{and}\ U_s=s-S_{N_{s}-1}\to\infty. \label{eq:cvinfty}
\end{equation}

Let
\begin{align}
 \ust(s,t)&=\P\left(\frac{U_s}{O_s+U_s}\leq t\right)= \P(tX_{N_s}>s-S_{N_s-1})\nonumber\\
& =\sum_{k=1}^{\infty}\P(X_{N_s}>(s-S_{N_s-1})/t,N_s=k)\nonumber\\
&=\sum_{k=1}^{\infty}\P(X_{k}>(s-S_{k-1})/t,M_{k-1}\leq s). \label{eq:sumOverUnder}
\end{align}

Note that 
\begin{align}\label{j11.1}
\ust(s,1)=1,
\end{align}
because
$\ust(s,1)$ is the probability that the random walk
$\{S_k\}$ will take a value greater than $s$ and by Lemma \ref{lema:l1} (c), this probability is 1.

\begin{lemma}\label{lema:OverUnderWH}
$\ust(s,t)$ is the minimal solution to the Wiener-Hopf equation
$$W(s,t)=P(X_1>s/t)+\int_{-\infty}^{s}W(s-y,t)\P_{X_1}(dy)\, dy.$$
\end{lemma}

\begin{proof}
Fix $t>0$ and let $g(s)=P(X_1>s/t)$. Formula $(\ref{eq:sumOverUnder})$
and the Markov property imply that $\ust(s,t)=\sum_{k=0}^{\infty}\E(g(s-S_{k})\1_{M_k<s})$.
By Theorem \ref{j10.1}, the function $\ust(s,t)$ is the minimal solution to the Wiener-Hopf equation.
\end{proof}

\begin{lemma}\label{prop:wheq}
The function 
\begin{align*}
\wt{\ust}(s,t) :=\ust(s,t)-t^2\ust(s,1) =\ust(s,t)-t^2
\end{align*}
 is a solution to the following equation
$$W(s,t)=P(X_1>s/t)-t^2P(X_1>s)+\int_{-\infty}^{s}W(s-y,t)f_X(y)\, dy.$$
Moreover, this is the minimal solution to this equation, that is, for every (positive) solution $W$ of this equation  we have
$\wt\ust(s,t)\leq W(s,t)$.
\end{lemma}

\begin{proof}
We have
\begin{align*}
 &\ust(s,t)-t^2\ust(s,1)\nonumber\\
&=\sum_{k=1}^{\infty}\P(X_{k}>(s-S_{k-1})/t,M_{k-1}\leq s)-\sum_{k=1}^{\infty}t^2\P(X_{k}>s-S_{k-1},M_{k-1}\leq s)\nonumber\\
&=\sum_{k=1}^{\infty}\P(X_{k}>(s-S_{k-1})/t,M_{k-1}\leq s)-t^2\P(X_{k}>s-S_{k-1},M_{k-1}\leq s)\\
&=\P(X_{1}>s/t,M_{k-1}\leq s)-t^2\P(X_{1}>s,M_{k-1}\leq s)\\
&\quad +\sum_{k=2}^{\infty}\P(X_{k}>(s-S_{k-1})/t,M_{k-1}\leq s)-t^2\P(X_{k}>s-S_{k-1},M_{k-1}\leq s).\\
\end{align*}
Setting $g(s):=\P(X_{1}>s/t)-t^2\P(X_{1}>s)$, we obtain,
\begin{align}\label{j11.11}
 &\P(X_{k}>(s-S_{k-1})/t,M_{k-1}\leq s)-t^2\P(X_{k}>s-S_{k-1},M_{k-1}\leq s)\\
&= \int_{-\infty}^s \P(X_k>(s-y)/t)-t^2\P(X_{k}>s-y)\P(S_{k-1}\in dy,M_{k-2}\leq s) \nonumber \\
&= \int_{-\infty}^s g(s-y)\P(S_{k-1}\in dy,M_{k-2}\leq s)\nonumber \\
&=\E(g(s-S_{k-1})\1_{(M_{k-1}\leq s)}).\nonumber 
\end{align}
Hence,
$\wt\ust(s,t)=\sum_{k=1}^{\infty}\E(g(s-S_{k-1})\1_{(M_{k-1}\leq s)})$, and from
$(\ref{swh:eq2})$ and Theorem \ref{j10.1} we know that this is the minimal solution to the  Wiener-Hopf equation in the statement of the lemma.
\end{proof}

\begin{lemma}\label{thm:cnv0}
For a fixed $t\in(0,1]$ we have 
$$\lim_{s\to\infty}\wt\ust(s,t)=\lim_{s\to\infty}\ust(s,t)-t^2\ust(s,1)=0.$$
\end{lemma}

\begin{proof}

By (\ref{eq3}), $P(X_1>s/t)-t^2P(X_1>s)\leq 4(1+s^2)^{-2}$.
Since $\int_0^{\infty}2(1+s^2)^{-1}\, ds<\infty$ and $s\mapsto 2(1+s^2)^{-1}$ is decreasing,
 Remark \ref{j8.5} (ii) shows that this function is directly Riemann integrable. This implies that $4(1+s^2)^{-2}$
is a product of two decreasing directly Riemann integrable functions. The lemma now follows from Lemmas \ref{thm:whems} and \ref{prop:wheq}. 
\end{proof}

\begin{lemma}

\label{cor:logdiff} 

(a) We have $\lim_{s\to \infty}\ust(s,t)=t^2$ for all $t\in [0,1]$. 

(b) Suppose that $Z$ has the distribution $\unif(0,1)$. The following limits hold in distribution, as $s\to \infty$,
\begin{align}\label{j11.2}
\frac{U_s}{O_s+U_s}&\to\sqrt{Z},\\
\frac{U_s}{O_s}&\to\frac{\sqrt{Z}}{1-\sqrt{Z}}.\label{j11.3}
\end{align}

(c)  The following limits hold in probability,
\begin{align*}
 \lim_{s\to\infty}\frac{\log U_s -\log O_s}{\log s}&=0,\\
 \lim_{s\to\infty}\frac{\log (U_s+O_s) -\log O_s}{\log s}&=0.
\end{align*}
\end{lemma}

\begin{proof}

(a) Recall from \eqref{j11.1}
that $\ust(s,1)=1$
and apply Lemma \ref{thm:cnv0}.

(b)
Since $t^2$ is the cumulative distribution function of $\sqrt{Z}$ and, by definition, $\ust(s,t)=\P\left(\frac{U_s}{O_s+U_s}\leq t\right)$,
\eqref{j11.2} follows.
The formula in \eqref{j11.3} follows easily from \eqref{j11.2}.

(c) Take the logarithm of the left hand side in \eqref{j11.2} (resp. \eqref{j11.3}) and divide by $\log s$. The logarithm of the right hand side of each \eqref{j11.2} and \eqref{j11.3}, divided by $\log s$, converges to 0 in distribution.
\end{proof}

\begin{corol}\label{OU:limits}
Suppose that $\P(X_1>t)=\frac{1}{2}-\frac{t}{2\sqrt{1+t^2}}$ and $Z$ has the distribution $\unif(0,1)$. The following limits hold in distribution, as $s\to \infty$,
\begin{align}\label{j11.6}
 \left(\sqrt{\frac{\log U_s}{\log s}},\sqrt{\frac{\log O_s}{\log s}}\right)&\to (Z,Z),\\
 \left(\sqrt{\frac{\log (U_s+O_s)}{\log s}},\sqrt{\frac{\log O_s}{\log s}}\right)&\to (Z,Z).\label{j11.7}
\end{align}
\end{corol}

\begin{proof}
The corollary follows from Lemma \ref{prop:limit_log} and Corollary \ref{cor:logdiff} (c).
\end{proof}

\subsection{Asymptotic independence of exit characteristics}

From the intuitive point of view, one would expect that when $s$, the distance from the light source  to the right edge of the strip, is large then the following random variables would be approximately  independent: the size of the undershoot, the ratio of the undershoot and overshoot, and the last side (upper of lower) visited by the light ray before the exit from the strip. We will prove that this is actually true. The idea of the proof is natural but its rigorous presentation requires extensive formulas.

\begin{lemma}\label{lemma:cnv:even:odd} 
For $t\in [0,1]$ and $j=0,1$, 
$$\lim_{s\to \infty}\P\left(\sqrt{\frac{\log U_s}{\log s}}\leq t, \1_{2\N}(N_s)=j\right)=\frac{1}{2}t.$$
\end{lemma}

\begin{proof}
We set 
$S_n=\sum_{k=1}^nX_k$ and $S_n'=\sum_{k=2}^nX_k$. We define $N_s= \inf\{n:S_n>s\}$, $N_s'= \inf\{n:S'_n>s\}$
and  $U_s=s-S_{N_s-1}$, $U_s'=s-S'_{N'_s-1}$. The  definitions of $O_s$ and $O'_s$ are analogous.
We have
\begin{align*}
 & \P\left(\frac{\log U_s}{\log s}\leq t^2, N_s\in 2\N\right) \\
&= \P\left(\frac{\log U_s}{\log s}\leq t^2, N_s\in 2\N, \min\{O_s,O_s',U_s,U'_s\}/2>|X_1|\right)\\
 &\quad+\P\left(\frac{\log U_s}{\log s}\leq t^2, N_s\in 2\N, \min\{O_s,O_s',U_s,U'_s\}/2\leq |X_1|\right).
\end{align*}
On the event $\min\{O_s,O_s',U_s,U'_s\}/2>|X_1|$ we have $U_s'=U_s-X_1$ and $N_s=N_s'+1$. Hence we have
\begin{align*}
 & \P\left(\frac{\log U_s}{\log s}\leq t^2, N_s\in 2\N\right) \\
 &\leq \P\left(\frac{\log (U_s'+X_1)}{\log s}\leq t^2, N_s'\in 2\N-1, \min\{O_s,O_s',U_s,U'_s\}/2>|X_1|\right)\\ 
&\quad+\P\left(\min\{O_s,O_s',U_s,U'_s\}/2\leq |X_1|\right)\\
&\leq \P\left(\frac{\log U_s'}{\log s}\leq t^2+\frac{\log (1+ X_1/U_s')}{\log s}, N_s'\in 2\N-1, \min\{O_s,O_s',U_s,U'_s\}/2>|X_1|\right)
\\ &\quad+\P\left(\min\{O_s,O_s',U_s,U'_s\}/2\leq |X_1|\right)\\
&\leq \P\left(\frac{\log U_s'}{\log s}\leq t^2+\frac{\log (3/2)}{\log s}, N_s'\in 2\N-1, \min\{O_s,O_s',U_s,U'_s\}/2>|X_1|\right)\\
&\quad+\P\left(\min\{O_s,O_s',U_s,U'_s\}/2\leq |X_1|\right)\\
&\leq \P\left(\frac{\log U_s'}{\log s}\leq t^2+\frac{\log (3/2)}{\log s}, N_s'\in 2\N-1\right)& 
\\&\quad+\P\left(\min\{O_s,O_s',U_s,U'_s\}/2\leq |X_1|\right).
\end{align*}
Since $(S_n)$ and $(S_n')$ have the same distribution,
$$\P\left(\frac{\log U_s'}{\log s}\leq t^2, N_s'\in 2\N-1\right)=\P\left(\frac{\log U_s}{\log s}\leq t^2, N_s\in 2\N-1\right).$$
Hence, subtracting from the previous inequality we get  
\begin{align*}
 &\P\left(\frac{\log U_s}{\log s}\leq t^2, N_s\in 2\N\right)-\P\left(\frac{\log U_s}{\log s}\leq t^2, N_s\in 2\N-1\right)\\
&\leq \P\left(t^2<\frac{\log U_s'}{\log s}\leq t^2+\frac{\log (3/2)}{\log s}, N_s'\in 2\N-1\right)+\P\left(\min\{O_s,O_s',U_s,U'_s\}/2\leq |X_1|\right)\\
&\leq \P\left(t^2<\frac{\log U_s'}{\log s}\leq t^2+\frac{\log (3/2)}{\log s}\right)+\P\left(\min\{O_s,O_s',U_s,U'_s\}/2\leq |X_1|\right).
\end{align*}
It follows from Lemma \ref{prop:limit_log} and \eqref{j11.6} that $\min\{O_s,O_s',U_s,U'_s\}\to \infty$ in probability as $s\to\infty$. 
The first term on the right hand side of the last formula goes to 0 as $s\to\infty$ in view of \eqref{j11.6}. Thus,
$$\limsup_{s\to\infty}\P\left(\frac{\log U_s}{\log s}\leq t^2, N_s\in 2\N\right)-\P\left(\frac{\log U_s}{\log s}\leq t^2, N_s\in 2\N-1\right) =0.$$
 We can show in a similar manner that
$$\limsup_{s\to\infty}\P\left(\frac{\log U_s}{\log s}\leq t^2, N_s\in 2\N-1\right)-\P\left(\frac{\log U_s}{\log s}\leq t^2, N_s\in 2\N\right) =0.
$$
The claim now follows from the fact that 
$$\lim_{s\to\infty}\P\left(\frac{\log U_s}{\log s}\leq t^2, N_s\in 2\N-1\right)+\P\left(\frac{\log U_s}{\log s}\leq t^2, N_s\in 2\N\right) =t.$$
\end{proof}

The following is one of our main results. 
\begin{thm}\label{thm:asymindeplim}
For $t, v\in [0,1]$, $j=0,1$, 
$$\lim_{s\to \infty}\P\left(\sqrt{\frac{\log U_s}{\log s}}\leq t, \frac{U_s}{U_s+O_s}\leq v, \1_{2\N}(N_s)=j\right)=\frac{1}{2}tv^2.$$
\end{thm}

\begin{proof}

First note that
\begin{align*}
 &\P\left(\sqrt{\frac{\log U_s}{\log s}}\leq t, \frac{U_s}{U_s+O_s}\leq v,\1_{2\N}(N_s)=0\right)\\
& = \P\left(U_s\leq s^{t^2}, \frac{U_s}{U_s+O_s}\leq v,\1_{2\N}(N_s)=0\right)\\ &= \P\left(s-S_{N_s-1}\leq s^{t^2}, \frac{s-S_{N_s-1}}{X_{N_s}}\leq v,\1_{2\N}(N_s)=0\right)\\
&=\sum_{k=1}^{\infty}\P\left(s-S_{N_s-1}\leq s^{t^2}, \frac{s-S_{N_s-1}}{X_{N_s}}\leq v,\ N_s=2k\right)\\
&=\sum_{k=1}^{\infty}\P\left(\frac{s-S_{2k-1}}{v}\leq X_{2k}, \ s- s^{t^2}\leq S_{2k-1}\leq s, S_{2k-2}\leq s,\ldots , S_1\leq s \right)\\
&=\int_{s-s^{t^2}}^s\P\left(\frac{s-u}{v}\leq X_{1}\right)\sum_{k=1}^{\infty}\P\left(  S_{2k-1}\in du, M_{2k-1}\leq s \right).
\end{align*}
We have 
\begin{align*}
A_s & := \P\left(\sqrt{\frac{\log U_s}{\log s}}\leq t, \frac{U_s}{U_s+O_s}\leq v, \1_{2\N}(N_s)=0\right)\\ 
&\quad - v^2\P\left(\sqrt{\frac{\log U_s}{\log s}}\leq t,\1_{2\N}(N_s)=0\right)\\ 
&= \int_{s-s^{t^2}}^s\left[\P\left(\frac{s-u}{v}\leq X_{1}\right)-v^2\P\left({s-u}\leq X_{1}\right)\right]\sum_{k=1}^{\infty}\P\left(  S_{2k-1}\in du, M_{2k-1}\leq s \right). \\
\end{align*}
Lemma \ref{lem:ineq} implies that $A_s\geq 0$. 
It follows from \eqref{j11.11} that 
\begin{align*}
& A_s\leq \int_{-\infty}^s\left[\P\left(\frac{s-u}{v}\leq X_{1}\right)-v^2\P\left({s-u}\leq X_{1}\right)\right]\sum_{k=1}^{\infty}\P\left(  S_{k-1}\in du, M_{k-1}\leq s \right)\\
 &=\wt{\ust}(s,v),
\end{align*}
where $\wt{\ust}$ is defined  in Proposition \ref{prop:wheq}. By Theorem \ref{thm:cnv0},
$\lim_{s\to\infty}\wt{\ust}(s,v)=0$.
The theorem now follows from Lemma \ref{lemma:cnv:even:odd}.
\end{proof}

We record a few variants and corollaries of the last theorem. They follow easily  from Lemma \ref{cor:logdiff} (c) and Theorem \ref {thm:asymindeplim}.

\begin{corol}
For $t, v\in [0,1]$, $j=0,1$, 
\begin{align}
&\lim_{s\to\infty}\P\left(\sqrt{\frac{\log O_s}{\log s}}\leq t, \frac{U_s}{U_s+O_s}\leq v, \1_{2\N}(N_s)=j\right)=\frac{1}{2}tv^2,\nonumber \\
 &\lim_{s\to \infty}\P\left(\sqrt{\frac{\log (U_s+O_s)}{\log s}}\leq t, \frac{U_s}{U_s+O_s}\leq v, \1_{2\N}(N_s)=j\right)=\frac{1}{2}tv^2,\label{eq:YNsrtio}\\
 &\lim_{s\to \infty}\P\left(\frac{U_s}{U_s+O_s}\leq v, \1_{2\N}(N_s)=j\right)=\frac{1}{2}v^2,\label{eq:ratioN}\\
&\lim_{s\to \infty}\P\left( N_s\in 2\N\right)=\lim_{s\to \infty}\P\left( N_s\in 2\N-1\right)=\frac{1}{2}.\label{eq:NEvenOdd}
\end{align}
\end{corol}

\subsection{Exit angle and position}
We introduced the exit angle $\Lambda_s$ and position $Y_s$ 
in Definition \ref{def:LPs}. Now we will describe their joint
distribution. 
Recall that $(s+S_n: 0\leq n\leq N_s-1)$ are $x$-coordinates of the reflection points of the light ray inside the tube before the exit time.
\begin{lemma}\label{lem:AnglePos}
For $s\geq 0$ we have 
\begin{equation}
 (\Lambda_s, Y_s)=
\begin{cases}
   \left(\cot^{-1}( O_s+U_s),\frac{U_s}{O_s+U_s}\right), & 
\text{if  } N_s\in 2\N-1,\\
\left(-\cot^{-1}( O_s+U_s),\frac{O_s}{O_s+U_s}\right), & \text{if  } N_s\in 2\N.
\end{cases} 
\label{eq:LPNs}
\end{equation}
\end{lemma}

\begin{proof}

If $N_s$ is even then the last reflection happened on the upper boundary of the tube and the angle is negative. 
Hence, $\Lambda_s =-\cot^{-1}( O_s+U_s)$. One can use similar triangles (see Figure \ref{fig:P_even}) to show that
 $Y_s=\frac{O_s}{O_s+U_s}$. 

\begin{figure}[ht]
\begin{center}
 \includegraphics[width=8cm]{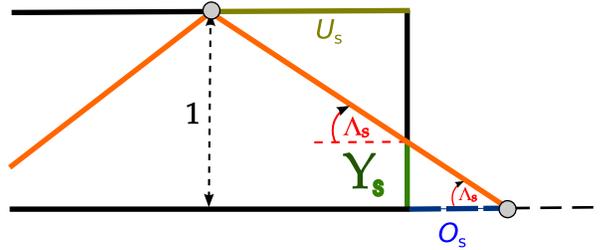}\\
\caption{The case when $N_s$ is even and $\Lambda_s$ is negative.}\label{fig:P_even}
\end{center}
\end{figure}

The case when $N_s\in 2\N-1$ can be dealt with in a similar way.
\end{proof}

\begin{thm}\label{thm:unif_law}
$(\Lambda_s,Y_s)\to (0,\unif[0,1])$ in distribution as $s\to \infty$.
\end{thm}

\begin{proof}
Recall from $(\ref{eq:cvinfty})$ that $O_s+U_s\to \infty$ in probability. Therefore 
$\Lambda_s\to 0$ in probability as $s\to \infty$. 

It remains to show that $Y_s\to \unif[0,1]$. We  use 
\eqref{eq:ratioN}, \eqref{eq:NEvenOdd} and \eqref{eq:LPNs} to see that
\begin{align*}
 \P(Y_s\leq t)
&=\P(Y_s\leq t,N_s\in 2\N-1)+\P(Y_s\leq t,N_s\in 2\N)\\
&= \P\left(\frac{U_s}{O_s+U_s}\leq t,N_s\in 2\N-1\right)+\P\left(\frac{O_s}{O_s+U_s}\leq t,N_s\in 2\N\right)\\
&= \P\left(\frac{U_s}{O_s+U_s}\leq t,N_s\in 2\N-1\right)\\
&\quad+\P(N_s\in 2\N)-\P\left(\frac{U_s}{O_s+U_s}\leq 1-t,N_s\in 2\N\right)\\
&\to\frac{t^2}{2}+\frac{1}{2}-\frac{(1-t)^2}{2}=t.
\end{align*}
\end{proof}

\begin{thm}\label{thm:scaling}
For $t,v\in [0,1]$, 
$$\lim_{s\to\infty}\P\left(\sqrt{\frac{\log \cot |\Lambda_s| }{\log s}}\leq t,j \Lambda_s \leq 0, Y_s\leq v\right)
=\begin{cases}
	t(1-(1-v)^2)/2,& \text{if  } j=1,\\
	tv^2/2,& \text{if  }  j=-1.
  \end{cases}
$$
\end{thm}

\begin{proof}
It follows from Lemma \ref{lem:AnglePos} that 
$\log \cot |\Lambda_s|=\log(O_s+U_s)$ and
$$\{j\Lambda_s\leq 0,Y_s\leq v\}=
\begin{cases}
\{N_s\in 2\N-1,\frac{U_s}{O_s+U_s}\leq v\},&  \text{if  } j=-1,\\
 \{N_s\in 2\N,1-\frac{U_s}{O_s+U_s}\leq v\},& \text{if  }  j=1.
 \end{cases}
$$ 
The theorem now follows from $(\ref{eq:YNsrtio})$.
\end{proof}

\begin{corol}\label{corol:eyedist}
For $t,y\in [0,1]$ and $\varepsilon\in (0,y)$ we have
$$\lim_{s\to\infty}\P\left(\sqrt{\frac{\log \cot |\Lambda_s| }{\log s}}\leq t, j\Lambda_s \leq 0\mid Y_s\in (y-\varepsilon,y ]\right)
=
\begin{cases}
    t(1-p+\varepsilon/2),&\text{if  }   j=1,\\
    t(p-\varepsilon/2),& \text{if  }  j=-1.
\end{cases}
$$
\end{corol}

\begin{proof}
Theorem \ref{thm:unif_law} shows that 
$\lim_{s\to\infty}\P\left( Y_s\in (y-\varepsilon,y ]\right)$ exists and is non-zero. This and Theorems \ref{thm:unif_law} and \ref{thm:scaling} can be used to derive the asymptotic  formula for the conditional probability.
\end{proof}

\subsection{Discussion of the results}\label{j12.1}
Theorem \ref{thm:unif_law} says that in the limit (i.e., when the light source  is 
infinitely far away), the light rays exit the two-dimensional tube horizontally, and they are equally likely
to exit at any point of the right edge.

Next we discuss the direction from which light rays arrive at an eye located at a point $(0,y)$
(see Figure \ref{fig:eye_im_beta}).
Corollary \ref{corol:eyedist} says that for large $s$,
$$\left(\sqrt{\frac{\log \cot |\Lambda_s| }{\log s}},\ \textrm{sgn}~\Lambda_s \right)\stackrel{d}{\approx} (V,R)$$
where $V$ has the uniform distribution $\unif[0,1]$ and $R$ is an independent random variable with $\P(R=1)=y$ and $\P(R=-1)=1-y$.
We can ``solve for $\Lambda_s$'' to derive the following purely heuristic formula, 
$$\Lambda_s \stackrel{d}{\approx} R\cot^{-1}\left(s^{V^2}\right).$$
Approximately $y$ proportion of light arrives from the lower side
(yellow rays in Figure \ref{fig:eye_im_beta}), while the remaining rays arrive from the upper side of the tube (orange rays 
in Figure \ref{fig:eye_im_beta}). 
The histogram in Figure \ref{fig:angle_dist} represents a simulation of 
$R\cot^{-1}\left(s^{V^2}\right)$.

\section{Three-dimensional model}\label{sec:cylinder}

This part of the paper will be devoted to light reflections within a three-dimensional semi-infinite cylinder $C=\{(x,y,z)\in \R^2\, :\, z^2+y^2=1, x\leq 0\}$ (see Figure \ref{sl3}.). In this case, the exiting light rays are not asymptotically parallel when the light source moves to infinity. So the three-dimensional model is less degenerate than the two-dimensional model. In this case, our results are  less complete than those  in the two-dimensional case. The reason is that deriving explicit formulas for this model is hard---this is a well known difficulty with models related to the Wiener-Hopf equation (see Section 6.5 of \cite{Kypr}, and especially subsection 6.5.4).
\begin{figure}[ht]
\begin{center}
 \includegraphics[width=7cm]{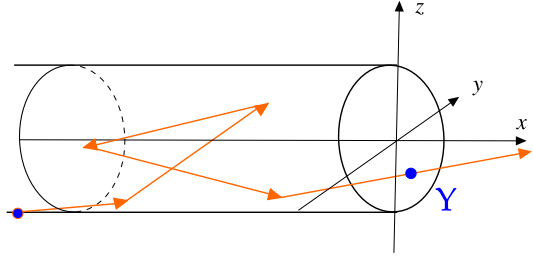}
\caption{Cylinder }\label{sl3}
\end{center}
\end{figure}

 We will assume that the light ray  starts at 
$\bs:=(-s,0,-1)$.
At the initial time and 
whenever the light ray hits the boundary of $ C$, it reflects according to the Lambertian scheme, i.e., 
\begin{enumerate}[(i)]
\item
the outgoing light ray forms a random angle $\Theta$ with the normal to the tangent plane,
\item 
$\Theta$ is a random variable with density (\ref{eq1}),
\item
 the projection of the outgoing ray onto 
the tangent plane forms a random angle $\Phi$ with the line parallel to the $x$-axis (see Figure \ref{3drefl}),
\item
$\Phi$ has the distribution $ \unif[-\pi/2,\pi/2]$
and is independent of $\Theta$.
\end{enumerate}
The consecutive reflection directions are jointly independent.

\begin{figure}[ht]
\begin{center}
 \includegraphics[width=8cm]{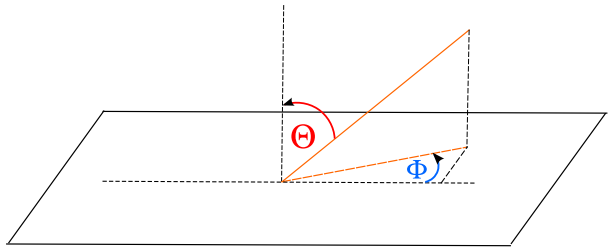}
\caption{Reflection with respect to the tangent plane}\label{3drefl}
\end{center}
\end{figure}

Consider the  light ray leaving the starting point $(x,0,-1)$. 
The tangent plane to the cylinder $C$ at that point is $\{z=-1\}$ and the 
ray starts moving along the line parallel to the vector 
\begin{equation}
 (R\sin\Theta\cos \Phi, R\sin\Theta \sin \Phi, R\cos\Theta),
\quad R>0 .\label{eq:reflvector}
\end{equation}

\begin{lemma}
Given $\Theta$ and $\Phi$, the distance to the next reflection point is
\begin{equation}
 R=\frac{2\cos\Theta}{\sin^2\Phi\sin^2\Theta+\cos^2 \Theta}.\label{eq:radius}
\end{equation}
\end{lemma}

\begin{proof}
We need to find a point $(R\sin\Theta\cos \Phi+x, R\sin\Theta \sin \Phi, R\cos\Theta-1)$
on the cylinder $\{y^2+z^2=1\}$. A straightforward calculation yields the formula. 
\end{proof}

\begin{lemma}
If the light ray reflects at the point $(x_0,y_0,z_0)$ and $\Theta$ and $\Phi$ are given then the next reflection point
will occur at  
$$(R\sin\Theta\cos \Phi +x_0, R(-z_0\sin\Theta \sin \Phi-y_0\cos\Theta), R(y_0\sin\Theta \sin \Phi-z_0\cos\Theta)),$$
where $R$ is given by $(\ref{eq:radius})$.
\end{lemma}

\begin{proof}
Note that $(x_0,y_0,z_0) = (x_0,0,-1)A$ where $A$ is the matrix representing the rotation operator
about the $x$-axis and given by 
$$A=
\left[\begin{array}{ccc}
         1 &0 & 0\\
	 0 & -z_0& -y_0 \\
	 0 & y_0 & -z_0
        \end{array}\right].$$
In view of \eqref{eq:reflvector}-\eqref{eq:radius},  if the light ray starts at
$(x_0,y_0,z_0)$, then the next reflection point will be 
at the point $(R\sin\Theta\cos \Phi+x_0, R\sin\Theta \sin \Phi, R\cos\Theta)A$.
\end{proof}

Next we establish  notation for the process  of reflection points inside the cylinder $C$.
Recall that the light ray starts at
$\bs=(-s,0,-1)$.
The reflection points will be $\{S_k+\bs\}_{k\geq 0}$ where $\{S_k\}_{k\geq 0}$
is  a random walk defined as follows.
Let $(\Theta_k,\Phi_k)_{k=1}^{\infty}$ be an i.i.d. sequence such that 
\begin{itemize}
 \item $\Theta_k$ has the density $\frac{1}{2}\cos \theta$ on $[-\pi/2,\pi/2]$ for all $k$;
 \item $\Phi_k$ is distributed as $\unif[-\pi/2,\pi/2]$ for all $k$;
 \item all random variables in the union of  the families $(\Theta_k)_{k\geq0}$ and $(\Phi_k)_{k\geq0}$  are independent.
\end{itemize}
Set $S_0 = (0,0,-1)$, 
$$R_{k}=\frac{2\cos\Theta_k}{\sin^2\Phi_k\sin^2\Theta_k+\cos^2 \Theta_k}=\frac{2\cos\Theta_k}{1-\cos^2\Phi_k\sin^2\Theta_k},$$
and define $S_k=(S^x_k,S^y_k,S^z_k)$ for $k\geq 1$ by 
\begin{align}
 S^x_{k}&=S^x_{k-1}+R_k\cos\Phi_k \sin\Theta_k, \label{eq:sx}\\
 S^y_{k}&=S^y_{k-1}+R_k(-S^z_{k-1}\sin\Phi_k\sin\Theta_k-S^y_{k-1}\cos\Theta_k),\label{eq:sy}\\
 S^z_{k}&=S^z_{k-1}+R_k(S^y_{k-1}\sin\Phi_k\sin\Theta_k-S^z_{k-1}\cos\Theta_k).\label{eq:sz}
\end{align}

\begin{remark}
(i)
The pair $(S^y_k, S^z_k)$ takes values in the unit circle $\{y^2+ z^2=1\}$. 
It is elementary to see that there exists $c_1>0$ such that for any $k$ and any  point $(y,z)$ on the unit circle, the conditional density of $(S^y_{k+2}, S^z_{k+2})$ with respect to the uniform probability measure, given $\{(S^y_k, S^z_k)= (y,z)\}$, is bounded below  by $c_1$ (note that the claim is about the distribution of $(S^y_{k+2}, S^z_{k+2})$, not $(S^y_{k+1}, S^z_{k+1})$).
A coupling argument now easily implies that the process $\{(S^y_k, S^z_k), k\geq 0\}$ is mixing and converges exponentially fast to a stationary distribution (which is necessarily uniform) on the unit circle. 

(ii)
The process $\{(S_k^x), k\geq 0\}$ is a random walk. 
Let $\prt_r C = \{(x,y,z)\in \R^2\, :\, y^2+z^2\leq 1, x= 0\}$,
$N_s=\inf\{n>0:S_n^x>s\}$ and let
 $Y_s \in \prt _r C$ denote the point where  the light ray crosses $\prt _r C$. 
It follows easily from
\eqref{lem:exp3dLddr} that the exit time $N_s$  goes to infinity as $s\to \infty$. This and (i) easily imply that the exit distribution on $\prt_r C$
is rotationally invariant.

\end{remark}

Let $X_k=S^x_k-S^x_{k-1}$ denote a step of the random walk $\{S^x_k\}$.
We will analyze the distribution of $X_k$. By \eqref{eq:sx},
\begin{equation}\label{eq:X_3d:def}
 X_k=\frac{2\cos\Theta_k \cos \Phi_k \sin \Theta_k}{\sin^2\Phi_k\sin^2\Theta_k+\cos^2 \Theta_k}.
\end{equation}

In order to simplify notation, we define $V_k=\sin \Theta_k$. By  \eqref{lem:theta}, $V_k$ has the distribution $\unif[-1,1]$.
Since the distribution of $\Theta_k$ is supported on $[-\pi/2,\pi/2]$, $\cos\Theta_k\geq 0$ and hence
$\cos\Theta_k = \sqrt{1-V_k^2}$. For the same reason $\cos \Phi_k\geq 0$.
This implies that 
\begin{equation}\label{eq:y:repr}
 X_k = \frac{2V_k\sqrt{1-V_k^2} \cos \Phi_k }{V_k^2\sin^2\Phi_k+1-V_k^2}= \frac{2V_k\sqrt{1-V_k^2} \cos \Phi_k }{1-V_k^2\cos^2\Phi_k}.
\end{equation}

\begin{lemma}\label{prop:step3d}
\begin{enumerate}[(a)]
  \item $\{X_k>0\}=\{V_k>0\}$  and $\{X_k<0\}=\{V_k<0\}$.
 \item $\E[X_k]=0$, $\E[|X_k|]=2-4/\pi$, $\E[X_k^2]=\pi/2$.
\end{enumerate}
\end{lemma}

\begin{proof}

(a) This part follows from (\ref{eq:y:repr}). 

\medskip

(b) We use $(\ref{eq:X_3d:def})$ to see that,
\begin{align*}
 \E[|X_k|]&=\int_{-\pi/2}^{\pi/2}\int_{-\pi/2}^{\pi/2}\frac{2\cos\theta \cos \phi |\sin \theta|}{\sin^2\phi\sin^2\theta+\cos^2 \theta}\, 
\frac 1 \pi \,d\phi\,\frac{\cos\theta}{2} d\theta\\
&=\frac 1 \pi \int_{0}^{\pi/2}\int_{0}^{\pi/2}\frac{8\cos\theta \cos \phi \sin \theta}{\sin^2\phi\sin^2\theta+\cos^2 \theta}\, d\phi\,\frac{\cos\theta}{2} d\theta\\
[h = \sin\phi]&=
\frac 1 \pi \int_{0}^{\pi/2}\int_{0}^{1}\frac{8\cos\theta \sin \theta}{h^2\sin^2\theta+\cos^2 \theta}\, dh\,\frac{\cos\theta}{2} d\theta\\ 
&=\frac 8 \pi \int_{0}^{\pi/2}\tan^{-1}(\tan \theta)\frac{\cos\theta}{2} d\theta\\ 
&=\frac 4 \pi \int_{0}^{\pi/2} \theta\cos\theta d\theta=2-4/\pi.
\end{align*}
Since $X_k$ is symmetric we have $\E X_k=0$.
A similar calculation yields $\E X_k^2=\pi/2$. 
\end{proof}

The following definitions are analogous to \eqref{j8.1}.
Recall that $S_0^x=0$, $S_n^x=\sum_{k=1}^nX_k$, and let
$Z_1^+=S^x_{T_1^+}$ and $Z_1^-=S^x_{T_1^-}$ where $T_1^+=\inf\{n: S^x_n>0\}$ and $T_1^-=\inf\{n: S^x_n<0\}$.
It follows from Corollary \ref{corol:swh1} and Lemma \ref{prop:step3d} (b) that 
\begin{align}\label{lem:exp3dLddr}
\E[Z_1^+]=\E[-Z_1^-]=\sqrt{\E[X_1^2]/2}=\sqrt{\pi}/2.
\end{align}

\begin{propo}\label{prop:limit3dU/O}
 For $t\in [0,1]$,
\begin{align}\label{u19.5}
\lim_{s\to\infty}\P\left(\frac{U_s}{U_s+O_s}\leq t\right)
= \Gamma(t) :=
\frac{\E[(t(U_0+O_0)-U_0)^+]}{\E[O_0]}.
\end{align}
\end{propo}
\begin{proof}
 We showed in Lemma \ref{lema:OverUnderWH} that $U(s,t)=\P\left(\frac{U_s}{U_s+O_s}\leq t\right)$
is the minimal solution to the Wiener-Hopf equation.

We have from Theorem \ref{whe:ms}
\begin{align*}
 \P\left(\frac{U_s}{U_s+O_s}\leq t\right)&= \P\left(X_{N_s}>\frac{s-S^x_{N_s-1}}{t}\right)\\
&= \int_0^sh(s-y)\, \ren^+(dy),
\end{align*}
where 
\begin{align*}
 h(s)&=\int_{-\infty}^0 \P\left(X_1>\frac{s-u}{t}\right)\ren^-(du)\\
     &=\sum_{k=1}^{\infty}\P\left(X_k>\frac{s-S^x_{k-1}}{t},S^x_{k-1}\leq 0,\ldots, S^x_0\leq 0\right)\\
    `    &=\P\left(X_{N_0}>\frac{s-S^x_{N_0-1}}{t}\right)=\P(t(U_0+O_0)-U_0>s).\\
\end{align*}
Since $s\mapsto\P(t(U_0+O_0)-U_0>s)$, \eqref{thm:rwn1} implies that
\begin{align*}
 \lim_{s\to\infty}\P\left(\frac{U_s}{U_s+O_s}\leq t\right)&=\lim_{s\to\infty}\int_0^s\P(t(U_0+O_0)-U_0>s-y)\ren^+(dy)\\
&=\frac{1}{\E[Z_1^+]}\int_0^\infty\P(t(U_0+O_0)-U_0>y)dy\\
&=\frac{\E[(t(U_0+O_0)-U_0)^+]}{\E[Z_1^+]}
=\frac{\E[(t(U_0+O_0)-U_0)^+]}{\E[O_0]}.
\end{align*}
\end{proof}

\begin{lemma}\label{lema:r:2}
The function $\Gamma:[0,1]\to \R$ defined in \eqref{u19.5}
has the following properties.
\begin{enumerate}[(a)]
 \item $\Gamma$ is a continuous, increasing and convex function.
 \item $ ( 2 \pi^{-1/2} - 4 \pi^{-3/2}) t\leq \Gamma(t)\leq t$ (note that 
$2 \pi^{-1/2} - 4 \pi^{-3/2}\approx 0.410$).
 \item $\Gamma(0)=0$, $\Gamma(1)=1$, $\Gamma'(0)=2 \pi^{-1/2} - 4 \pi^{-3/2}$ and $\Gamma'(1)=\frac{1}{\E O_0}\E[O_0+U_0]$.
\end{enumerate}
\end{lemma}

\begin{proof}
(a) Since $U_0$ and $O_0$ are non-negative, it is clear that $\Gamma$ is
an increasing and continuous function. A function of the form 
$t\mapsto (at-b)^+$ is a convex function for non-negative $a$ and $b$. 
Since $t\to \E[(t(U_0+O_0)-U_0)^+]$ is the expected value of convex functions, it is convex.

\medskip
(b) By the definition,
$\Gamma(t)=\frac{1}{\E O_0}\E[(tO_0-(1-t)U_0)^+]$. Since $-(1-t)U_0\leq 0$
we have $\Gamma(t)\leq \frac{1}{\E O_0}\E(tO_0)=t$. On the other hand, $\{U_0=0\}=\{X_1>0\}$,
a.s., and on that event we have $O_0=X_1$. Hence,
\begin{align}\label{u19.1}
 \Gamma(t)&=\frac{1}{\E O_0}\E[(tO_0-(1-t)U_0)^+(\1_{(U_0=0)}+\1_{(U_0>0)})]\\
&\geq \frac{1}{\E O_0}\E[(tO_0-(1-t)U_0)^+\1_{(U_0=0)}]=\frac{1}{\E O_0}\E[(tO_0)^+\1_{(U_0=0)}]\nonumber\\
&=\frac{1}{\E O_0}\E[(tX_1)^+\1_{(X_1>0)}]=t\cdot \frac{\E[X_1\1_{(X_1>0)}]}{\E O_0}. \nonumber
\end{align}
It follows from the symmetry of $X_1$ and Lemma \ref{prop:step3d} (b) that $\E[X_1\1_{(X_1>0)}]=\frac{1}{2}\E[|X_1|]=1-2/\pi$. 
Corollary \ref{corol:swh1} and Lemma \ref{prop:step3d} (b) imply that
$\E[O_0]=\E Z_1^+=\sqrt{\pi}/2$. 
\begin{align}\label{u19.2}
\frac{\E[X_1\1_{(X_1>0)}]}{\E O_0} = 2 \pi^{-1/2} - 4 \pi^{-3/2}.
\end{align}
This and \eqref{u19.1} imply part (b).

\medskip
(c) It is clear that $\Gamma(0)=0$ and $\Gamma(1)=1$. For the derivative at $t=0$ we have:
\begin{align*}
 \Gamma'(0)&=\lim_{t\to 0^+}\frac{\Gamma(t)-\Gamma(0)}{t}=\lim_{t\to 0^+}\frac{\Gamma(t)}{t}\\
      &=\lim_{t\to 0^+}\frac{1}{t\E O_0}\E[(tO_0-(1-t)U_0)^+(\1_{(U_0=0)}+\1_{(U_0>0)})]\\
      &=\lim_{t\to 0^+}\frac{1}{t\E O_0}\E[(tO_0)^+\1_{(U_0=0)}]+\lim_{t\to 0^+}\frac{1}{t\E O_0}\E[(tO_0-(1-t)U_0)^+\1_{(U_0>0)}]\\
      &=\frac{\E[X_1\1_{(X_1>0)}]}{\E O_0}+\lim_{t\to 0^+}\frac{1}{\E O_0}\E\left[\left(O_0-\frac{1-t}{t}U_0\right)^+\1_{(U_0>0)}\right].
\end{align*}
Note that $\left(O_0-\frac{1-t}{t}U_0\right)^+\1_{(U_0>0)}\leq O_0$ and $\lim_{t\to 0^+}\left(O_0-\frac{1-t}{t}U_0\right)^+\1_{(U_0>0)}=0$. Hence,
by dominated convergence 
and \eqref{u19.2},
\begin{align*}
\Gamma'(0)=\frac{\E[X_1\1_{(X_1>0)}]}{\E O_0}=2 \pi^{-1/2} - 4 \pi^{-3/2}.
\end{align*}
For the derivative at $t=1$ we have:
\begin{align*}
 \Gamma'(1)&=\lim_{t\to 1^-}\frac{\Gamma(1)-\Gamma(t)}{1-t}= \lim_{t\to 1^-}\frac{1}{\E O_0}\E\left[\frac{1}{1-t}O_0-\left(\frac{t}{1-t}O_0-U_0\right)^+\right].
\end{align*}
Since
$t\mapsto\frac{1}{1-t}O_0-\left(\frac{t}{1-t}O_0-U_0\right)^+$
is an increasing function, the monotone convergence theorem yields
\begin{align*}
 \Gamma'(1)= \frac{1}{\E O_0}\E[O_0+U_0].
\end{align*}
\end{proof}

Let $Y_s = (Y^x_s, Y^y_s, Y^z_s) \in \prt _r C$ denote the point through which the light ray exits the cylinder (see Figure \ref{3dOvrUnd}) and let
$N_s=\inf\{n>0:S_n^x>s\}$. Note that
\begin{align*}
 Y_s=\left(0,\frac{s-S^x_{N_s-1}}{X_{N_s}}\cdot (S_{N_s}^y-S^y_{N_s-1})+S^y_{N_s-1},\right. \nonumber
 \left. \frac{s-S^x_{N_s-1}}{X_{N_s}}\cdot (S_{N_s}^z-S^z_{N_s-1})+S^z_{N_s-1}\right).
\end{align*}

\begin{figure}[ht]
\begin{center}
 \includegraphics[width=10cm]{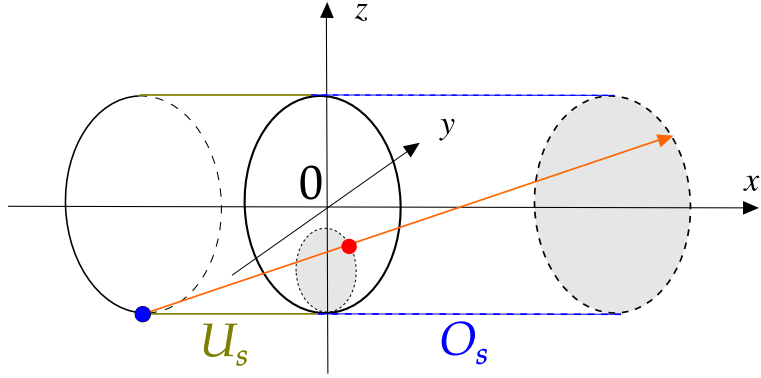}
\caption{Exit point $Y_s$ is marked with (\textcolor{red}{$\bullet$}).}\label{3dOvrUnd}
\end{center}
\end{figure}

It is elementary to derive the following formula from the above
definition of $Y_s$.
\begin{align}
\left\{\frac{U_s}{U_s+O_s}\leq t\right\}
&=\left\{(Y^y_s-(1-t)S^y_{N_s-1})^2+(Y^z_s-(1-t)S^z_{N_s-1})^2\leq t^2\right\}.
\label{u19.3}
\end{align}

Let $D_t(y_0,z_0)= \{(y,z): (y-(1-t)y_0)^2+(z-(1-t)z_0)^2\leq t^2 \}$.
Note that  $D_t(y_0,z_0)$ is a disc with area $\pi t^2$.
The definition of $D_t(y_0,z_0)$, \eqref{u19.5} and \eqref{u19.3} yield
\begin{align}\label{u19.4}
\lim_{s\to \infty} \P ( Y_s \in D_t(S^y_{N_s-1}, S^z_{N_s-1}) )
= \Gamma(t).
\end{align}
 
\begin{propo} \label{thm:limit:3d:from:direction}
Let $\leb(A)$ denote the Lebesgue measure on $\prt _r C$.
\begin{enumerate}[(a)]
\item
\begin{align*}
\lim_{t\to 0^+}\frac{\lim_{s\to \infty} \P ( Y_s \in D_t(S^y_{N_s-1}, S^z_{N_s-1}) )}{\leb(D_t(S^y_{N_s-1}, S^z_{N_s-1}))}=\infty.
\end{align*}
\item
\begin{align*}
\lim_{t\to 1^-}\frac{\lim_{s\to \infty} \P ( Y_s \in \prt _r C \setminus  D_t(S^y_{N_s-1}, S^z_{N_s-1}))}{\leb(\prt _r C \setminus D_t(S^y_{N_s-1}, S^z_{N_s-1}))}=\frac{1}{2\pi\E[O_0]}\E[O_0+U_0].
\end{align*}
\end{enumerate}
\end{propo}

\begin{proof}
(a)
We have 
$\leb(D_t(S^y_{N_s-1}, S^z_{N_s-1}))=\pi t^2$ so part (a) follows from Lemma \ref{lema:r:2} (b) and \eqref{u19.4}.

\medskip
(b) By \eqref{u19.4},
\begin{align*}
\lim_{t\to 1^-}\frac{\lim_{s\to \infty} \P ( Y_s \in \prt _r C \setminus  D_t(S^y_{N_s-1}, S^z_{N_s-1}))}{\leb(\prt _r C \setminus D_t(S^y_{N_s-1}, S^z_{N_s-1}))}
&=\lim_{t\to 1^-}\frac{1-\Gamma(t)}{\pi(1-t^2)}\\
&=\lim_{t\to 1^-}\frac{1}{\pi(1+t)}\frac{\Gamma(1)-\Gamma(t)}{1-t}=\frac{\Gamma'(1)}{2\pi}.
\end{align*}
Part (b) now follows from Lemma \ref{lema:r:2} (c).
\end{proof}

Let $B_r(0)=\{(y,z)\in \prt_r C:y^2+z^2\leq r^2\}$.

\begin{thm}\label{j12.7}
For $r\in (0,1)$,
\begin{align}\label{j12.8}
\lim_{s\to \infty} \P ( Y_s \in B_r(0))&\leq \Gamma\left(\frac{1+r}{2}\right)-\Gamma\left(\frac{1-r}{2}\right),\\
\lim_{s\to \infty} \P ( Y_s \in B_r(0))&\leq 
a(r) :=
( 1/2 + \pi^{-1/2} - 2 \pi ^{-3/2}) r 
+  1/2 - \pi^{-1/2} + 2 \pi ^{-3/2}.\label{j12.9}
\end{align}
where  $\Gamma$ is given by $(\ref{u19.5})$.

\end{thm}

\begin{remark}
The linear function $a(r)$
takes values $a(0)  \approx 0.29$ and $a(1) = 1$.
\end{remark}

\begin{proof}[Proof of Theorem \ref{j12.7}]
Since 
$B_r(0) \subset D_{(1+r)/2}(y_0,z_0)\setminus D_{(1-r)/2}(y_0,z_0)$ for any $(y_0,z_0)$ on the unit circle, we obtain
\eqref{j12.8} by applying \eqref{u19.4}.
 We derive \eqref{j12.9} from \eqref{j12.8} and Lemma \ref{lema:r:2} (b).
\end{proof}

\subsection{Brightness singularity}

We will show that the apparent brightness of the light 
arriving at the eye placed at the center of the tube opening $\prt_r C$
goes to infinity close to the center of the
field of vision as the light source moves to infinity. The precise
statement of the result is the following.

Let $\bv_s = (S_{N_s} - S_{N_s-1})/|S_{N_s} - S_{N_s-1}| $ be the unit
vector representing the direction of the light ray at the exit time.
Let $\cB(r) = \{(x,y,z): x^2 + y^2 + z^2 = 1, y^2+ z^2 \leq r^2 , x>0 \}$ denote a ball on the unit sphere and recall that $B_r(0)=\{(y,z)\in \prt_r C:y^2+z^2\leq r^2\}$.

\begin{thm}\label{thm:u20_1}
For any $0 < r_1 < r_2 < 1$,
\begin{align*}
\lim_{s\to \infty} \lim_{\delta\to 0}
\frac s {\pi \delta^2} \P\left(\bv_s \in \cB\left( \frac{r_2}{s}\right) \setminus \cB\left( \frac{r_1}{s}\right) , Y_s \in B_\delta(0)\right) 
=   \frac {r_2 - r_1}{2 \pi^2}.
\end{align*}
\end{thm}

The proof of the theorem will be preceded by a lemma. The lemma is 
an estimate for the Green function of the random walk $S^x_k$. The estimate is rather standard and it is likely to be known but we could not find
a ready reference.
Let $M_s(x_1, x_2) $ be the number of $k\leq N_s -1$ such that $S^x_k -s \in (x_1, x_2)$.

\begin{lemma}\label{u22.2}
For any $0 < a_1 < a_2 < 1$, 
\begin{align*}
\lim _{s\to \infty} \frac 1 {s^2} \E M_s(-s a_2, -s a_1)
= \frac{a^2_2 - a^2_1}{\pi}.
\end{align*}

\end{lemma}

\begin{proof}

Let $D[0,\infty)$ denote the space of RCLL functions equipped with the Skorokhod topology. Some of the functions in this space can be ``killed.'' We formalizing this idea by adding a ``coffin'' absorbing state to the state space and sending there all killed functions. We will use the convention that all functions take value 0 on the coffin state.

Let $\{W^s_t, t\geq 0\}$ be the one dimensional Brownian motion with $W^s_0=-s$. Let
$\tau_0 = \inf\{t\geq 0: W^s_t = 0\}$ and let $G_s(x)$ denote the Green function of $W^s$ killed at time $\tau_0$, i.e., $G_s$ is the function defined by the requirement that for all $-\infty < x_1 < x_2 < \infty$,
\begin{align*}
\int_{x_1}^{x_2} G_s(x) dx = \E \int_0^{\tau_0} \1_{(x_1, x_2)} (W^s_t) dt.
\end{align*}
It is standard to show that
\begin{align}\label{u21.1}
G_s(x) = 
\begin{cases}
-s & x < -s, \\
-x & -s \leq x \leq 0, \\
0 & x > 0.
\end{cases}
\end{align}

Recall from Lemma \ref{prop:step3d} (b) that $\E X_k =0$ and  $\E X^2_k =\pi/2$.
According to the Skorokhod embedding theorem (see \cite{Obloj}), there exist stopping times $T_k$, $k\geq 0$, such that $T_0 = 0$, $\{T_k - T_{k-1}, k\geq 1\}$ is an i.i.d. sequence with $\E(T_k - T_{k-1}) = \pi/2$, and $\{W^s_{T_k} - W^s_{T_{k-1}}, k\geq 1\}$ are i.i.d. with the same distribution as $\{X_k, k\geq 1\}$.
Let $N^W_s = \inf\{k\geq 0: W^s(T_k) > 0\}$ and
let $M^W_s(x_1, x_2) $ be the number of $k\leq N^W_s -1$ such that $W^s(T_k) \in (x_1, x_2)$.
It will suffice to prove the lemma for $M^W_s$ in place of $M_s$.

Fix an arbitrarily small $\eps>0$.
Let $c_1>0$ be so large that, for all $k$, a.s.,
\begin{align}\label{u21.2}
\E\left( \int_{T_{k-1}}^ {T_k} \1_{(W^s(T_{k-1})-c_1,W^s(T_{k-1})+c_1)}(W^s_t) dt
\mid \F_{T_{k-1}}
\right)\geq \pi/2 -\eps.
\end{align}
Suppose that $0 < a_1 < a_2 < 1$ and $s $ is so large that
$(-s a_2-c_1, -s a_1+c_1) \subset (-s, 0)$.
Then \eqref{u21.1} and \eqref{u21.2}  show that
\begin{align*}
&(\pi/2 -\eps)\E M^W_s(-s a_2, -s a_1)
=
\E \sum_{k=1}^{N^W_s} 
 \left(\1_{W^s(T_{k-1}) \in (-s a_2, -s a_1)} (\pi/2 -\eps)
\right)\\
&= 
\sum_{k=1}^\infty 
\E \left(\1 _{k \leq N^W_s}\1_{W^s(T_{k-1}) \in (-s a_2, -s a_1)} (\pi/2 -\eps)
\right)\\
&\leq 
\sum_{k=1}^\infty 
 \E\Bigg(\1 _{k \leq N^W_s} \1_{W^s(T_{k-1}) \in (-s a_2, -s a_1)} \\
& \qquad \times \E\left( \int_{T_{k-1}}^ {T_k} \1_{(W^s(T_{k-1})-c_1,W^s(T_{k-1})+c_1)} (W^s_t) dt
\mid \F_{T_{k-1}}
\right)
\Bigg)\\
&= 
\sum_{k=1}^\infty 
 \E\left(\1 _{k \leq N^W_s} \1_{W^s(T_{k-1}) \in (-s a_2, -s a_1)} \int_{T_{k-1}}^ {T_k} \1_{(W^s(T_{k-1})-c_1,W^s(T_{k-1})+c_1)} (W^s_t) dt
\right)\\
& \leq
\E \int_0^{T_{N^W_s}} \1_{(-s a_2-c_1, -s a_1+c_1)} (W^s_t) dt\\
&\leq
\E \int_0^{\tau_0} \1_{(-s a_2-c_1, -s a_1+c_1)} (W^s_t) dt
+ \E \int_{T_{N^W_s}-1}^ {T_{N^W_s}} \1_{(-s a_2-c_1, -s a_1+c_1)} (W^s_t) dt\\
&\leq 
\int_{-s a_2-c_1}^{-s a_1+c_1} G_s(x) dx
+ \E \left( T_{N^W_s}- T_{N^W_s-1} \right) \\
&= \frac 1 2 ( (-s a_2+c_1)^2 - (-s a_1-c_1)^2)
 + \pi/2 .
\end{align*}
It follows that
\begin{align*}
\limsup _{s\to \infty} \frac 1 {s^2} \E M_s(-s a_2,- s a_1)
\leq \frac{\frac 1 2 (a^2_2 - a^2_1)}{\pi/2 -\eps}.
\end{align*}
Since $\eps>0$ is arbitrarily small,
\begin{align}\label{u21.7}
\limsup _{s\to \infty} \frac 1 {s^2} \E M_s(-s a_2, -s a_1)
\leq \frac{a^2_2 - a^2_1}{\pi}.
\end{align}

Let 
\begin{align*}
V^s_t = \frac{S^x_{k}}{ s\sqrt{\pi^2/2}},\qquad 
\text{  for  } t\in [k /s^2, (k+1)/s^2),\  k\geq 0.
\end{align*}
The processes $\{V^s_t, t\geq 0\}$ converge to $\{W^1_t, t\geq 0\}$ in distribution in the Skorokhod topology when $s\to \infty$.

For $a,b$ and $\delta>0$ such that $a+\delta<  b-\delta$, choose some  continuous function $\lambda(a,b,\delta,t): \R \to [0,1]$  such that
\begin{align*}
\lambda(a,b,\delta, t) =
\begin{cases}
0, & \text{  if  } t < a, \\
1, & \text{  if  } a+\delta< t < b-\delta, \\
0, & \text{  if  } t >b .
\end{cases}
\end{align*}
Fix any $u < \infty$.
 The functional $f \to \int_0^u \lambda(a,b,\delta,f(t)) dt$ 
is  bounded and  continuous  on $D[0,\infty)$ in the Skorokhod topology. It follows that 
\begin{align*}
\liminf_{s\to\infty}
\E \left(\frac 1 {s^2}  M_s(-s a_2, -s a_1) \right)
&\geq
\liminf_{s\to\infty}
\E \int_0^u \lambda\left(-a_2/\sqrt{\pi/2}, - a_1\sqrt{\pi/2},\delta,V^s_t\right) dt\\
&= \E \int_0^u \lambda\left(-a_2/\sqrt{\pi/2}, - a_1\sqrt{\pi/2},\delta,W^1_t\right) dt \\
&\geq \int_{-a_2/\sqrt{\pi/2} + \delta}^{- a_1\sqrt{\pi/2} - \delta} G_s(x) dx \\
&= \frac 1 2 \left(\left(-a_2/\sqrt{\pi/2} + \delta\right)^2
- \left(- a_1\sqrt{\pi/2} - \delta\right)^2 \right).
\end{align*}
Since $\delta >0$ can be arbitrarily small,
\begin{align*}
\liminf_{s\to\infty}
\E \left(\frac 1 {s^2}  M_s(-s a_2, -s a_1) \right)
&\geq  \frac{a^2_2 - a^2_1}{\pi}.
\end{align*}
This and \eqref{u21.7} prove the lemma.
\end{proof}

\begin{proof}[Proof of Theorem \ref{thm:u20_1}]

 Recall that $B_r(0)=\{(y,z)\in \prt_r C:y^2+z^2\leq r^2\}$.
If $s$ is large, $\delta>0$ is small and the light ray leaves a point $(-s a , 0 ,1)$ in the random direction determined by \eqref{eq:sx}-\eqref{eq:sz} then the probability that this ray  will exit the tube through $B_\delta(0)$ is 
\begin{align}\label{u22.1}
\pi \delta^2 \frac 1 {2  \pi s a} \frac 1 {2 (s a)^2} + o(\delta^2/s^3)
=
 \frac {\delta^2} {4 s^3 a^3} + o(\delta^2/s^3) .
\end{align}
The factors on the left hand side represent the following quantities. The area of  
$B_\delta(0)$ is equal to $\pi \delta^2$. The hitting density is the product of two factors, corresponding to $\Phi$ and $\Theta$ (see the beginning of Section \ref{sec:cylinder} for the definitions). The factor representing the  density of $\Phi$ is $\frac 1 {2  \pi s a}$ (the reciprocal of the radius of the circle centered at the starting point and passing through the center of $\prt _r C$, up to a term of lower order). The factor representing the density of $\Theta $ is $\frac 1 {2 (s a)^2}$ because of \eqref{eq1} and scaling by the radius $sa$, just like in the case of the density  of $\Phi$; once again, the terms of lower order are ignored. 

For fixed $r_1$ and $ r_2$, large $s$, and small $\delta>0$,
a light ray arriving at $B_\delta(0)$ in the direction $\bv_s \in \cB\left( \frac{r_2}{s}\right) \setminus \cB\left( \frac{r_1}{s}\right)$ must have have left the surface of the tube at a point with the $x$-coordinate in the range
$ (-s/r_1 + o(s) + O(\delta s), -s/r_2 + o(s) + O(\delta s)) $.

We define a measure $\M_s$ by $\M_s(A)=\E\left(\sum_{k=0}^{N_s-1}\1_A(S^x_k)\right)$
for every Borel subset $A$ of $\R$. Lemma \ref{u22.2} can be rephrased as
$$\lim_{s\to\infty}\frac{1}{s^2}\M_s((-s a_2, -s a_1))=\frac{a^2_2 - a^2_1}{\pi}.$$
A formal calculation based on this formula yields for small $\eps>0$,
\begin{align}\label{jul2.1}
 \M_s(-s (a+\eps), -s a)
\approx \frac{ 2 a \eps s^2}{\pi}.
\end{align}
It is routine, using techniques from the proofs of Lemmas \ref{lema:3.3} and \ref{lema:lstep}, to provide a rigorous argument based on  \eqref{u22.1} and \eqref{jul2.1}, showing that
for any $0 < r_1 < r_2 < 1$,
\begin{align*}
&\lim_{s\to \infty} \lim_{\delta\to 0}
\frac s {\pi \delta^2} \P\left(\bv_s \in \cB\left( \frac{r_2}{s}\right) \setminus \cB\left( \frac{r_1}{s}\right) , Y_s \in B_\delta(0)\right)\\
=&\lim_{s\to \infty} \lim_{\delta\to 0}
\frac s {\pi \delta^2} \int_{-1/r_1 + o(s)/s + O(\delta s)/s}^{-1/r_2 + o(s)/s + O(\delta s)/s}\left[\frac {\delta^2} {4 s^3 a^3} + o(\delta^2/s^3)\right]
\M(s\, da)\\
=&
\lim_{s\to\infty}\left(s
\int_{1/r_2}^{1/r_1}  \frac 1 {4 \pi s^3 a ^3} 
\frac{2 a s^2}{\pi} da \right)= \frac {r_2 - r_1}{2 \pi^2}.
\end{align*}
\end{proof}

\subsection{Discussion of the results}

The behavior of the light reflection process in the three dimensional tube is much different from that in the two dimensional case.
The most notable difference is that the overshoot $O_s$ and undershoot $U_s$
(in the $x$-direction) converge to a non-trivial distribution (instead of
going to infinity as  in (\ref{eq:cvinfty})). The reason  is  that  the ladder variable $Z^+_1$ has finite expectation (see \eqref{lem:exp3dLddr}), unlike in the two dimensional case. This fact and the Wiener-Hopf equation 
can be used to show existence of the limiting distributions for many quantities 
of interest. Unfortunately
most of the formulas that can be obtained in this way are abstract integrals that cannot be easily 
interpreted. 

Theorem \ref{thm:u20_1} says that small annuli at the center of the field of vision,
with the area of magnitude $1/s^2$, receive about $1/s$ units of light. Hence, the apparent brightness is about $s$ at the distance $1/s$ from the center, if the light source is $s$ units away. 
This means that the surface of the tube does not appear to the eye to be Lambertian, i.e., the surface does not have uniform apparent brightness. This can be explained by the fact that not all parts of the surface of the tube  receive the  same amount of light.

\section*{Acknowledgments}
The authors would like to thank Ronald Doney, Andreas Kyprianou, Donald Marshall and Douglas Rizzolo for very helpful advice.
The first author is grateful to the Isaac Newton Institute for Mathematical Sciences, where this research was partly carried out, for the hospitality and support. The authors thank the anonymous referee for the detailed reading of the paper and many suggestions for improvement. 

\bibliographystyle{alpha}

\end{document}